\documentclass[12pt,a4paper,reqno]{amsart}
\usepackage{graphicx}
\usepackage{epsfig}
\usepackage{amsmath}
\usepackage{amsfonts}
\usepackage{amssymb}
\usepackage{amscd}
\newtheorem{theorem}{Theorem}[section]

\newtheorem{lemma}[theorem]{Lemma}

\numberwithin{equation}{section}

\begin{document}

\address{Nigar Aslanova  \newline
 Institute of Mathematics and Mechanics of  National Academy of Sciences of Azerbaijan, Baku, Azerbaijan\\
 Azerbaijan University of Architecture and Construction,  }
\email{nigar.aslanova@yahoo.com}

\address{khalig Aslanov  \newline
 Azerbaijan State Economic University (UNEC)}
\email{xaliqaslanov@yandex.ru}

\begin{center}
{{\bf ON EXTENSIONS AND SPECTRAL PROBLEMS FOR FOURTH ORDER DIFFERENTIAL OPERATOR EQUATION
}}
\end{center}
\

\begin{center}
\textbf{Nigar Aslanova, Kh. Aslanov}
\end{center}

\

\textbf{Abstract} \textit{ The aim of the paper is firstly to study domains of definitions in terms of boundary conditions of minimal and maximal operators, as well as selfadjoint extensions of a minimal operator associated with the fourth-order differential operator equation. Further, we give necessary and sufficient conditions for that operators to have a purely discrete or continuous spectrum, to exist extension with resolvent from $\sigma_p,$ study asymptotics of spectrum in case pure discrete spectrum. Finally, give the new and more general method for evaluations of regularized traces of operators with discrete spectrum associated with one class boundary value problems. }

\textbf{Keywords:} \textit{ Hilbert space, differential operator equation, selfadjoint extensions with exit from space, spectrum, eigenvalues, trace class operators, regularized trace.}

\textbf{Mathematics Subject Classification:} 34B05, 34G20, 34L20, 34L05, 47A05, 47A10.

\section{Introduction}\label{Sec:1}

Our aim is first  to study domains of definition of  minimal and maximal operators generated by a differential operator expression  in a  space  which  is larger   than one  where the differential expression is considered. Such operators arise upon consideration of boundary value  problems for differential equations when boundary conditions also contain an eigenvalue parameter. Secondly, to give boundary conditions for defining selfadjoint extensions, extensions with a discrete or continuous spectrum. Thirdly, to derive an asymptotic formula for a spectrum in the case of purely discrete spectrum, and finally, to give a  new method for finding regularized trace of the operator associated with the corresponding boundary value problem in one special case. We show a new treat for the deriving  the trace formula, which is more general in comparison with one applied  in our previous works and might be applied in future studies.

 Consider in $L_{2}H,(0,1))$,  where $H$ is an abstract separable Hilbert space, the following  differential expression with operator coefficients
\begin{equation}
ly \equiv y^{IV}(t) + Ay(t) + q(t)y(t)
\end{equation}

Here $A$ and $q(t)$ are operator coefficients. Our assumptions about them are the followings (later,when deriving trace formula we will put some additional requirements on $A$ and $q(t)):$

1. A is a selfadjoint operator in  $H$, moreover $A>I$  , where $I$ is an identity operator in $H$, and $A^{-1}\in {\sigma }_\infty$.

The last  condition provides discreteness of the spectrum of $A$.

 2. $q(t)$ is a weakly measurable, selfadjoint, bounded operator -valued function in H  for  each $t\in [0, 1]$.

So, $q(t)$ is bounded in $H$, while the operator $A$ is bounded only from below. Under that conditions $q(t)$ is bounded also in $L_{2} (H,(0,1)$.

 Consider the direct sum space $H_{1} = L_{2} (H,(0,1)) \oplus H^2_Q$ with the elements $Y= (y(t), y_{1} , y_{2}),$ $Z= (z(t), z_{1}, z_{2})$, where $y_1,y_2,z_1,z_2  \in H$. A scalar product in $H_1$ is defined by
\begin{equation}
{\left(Y,Z\right)}_1={\left(y\left(t\right),z(t)\right)}_{L_2(H,(0,1))}+  \left({Q_1}^{-\frac{1}{2}}y_1,\ {Q_1}^{-\frac{1}{2}}z_1\right)+\left({Q_2}^{-\frac{1}{2}}y_2,\ {Q_2}^{-\frac{1}{2}}z_2\right),
 \end{equation}
$\left(\cdot,\cdot \right)$ is a scalar product in $H$,  $Q_1$ and $Q_2$ are  self-adjoint positive -definite operators in $H$.

Define in $H_1$ an operator ${L'}_0$ in the following way:
\[D\left({L^{'}}_0\right)=\left(Y\backslash Y=\left\{y\left(t\right),Q_1y\left(1\right),Q_2y^{'}\left(1\right),\ \right\},y\left(t\right)\in C^{\infty }_0\left(H_\infty,\left(0,1\right]\right),\ \right.
\]
\[
\left.y\left(1\right)\in D(Q_1),\ y^{'}(1)\in {D(Q}_2), L_0^{'}Y=\{ly,-y^{III}(1),y'(1)\},\; q(t)\equiv 0\right)
\]
where $C^{\infty }_0(H_{\infty },\ (0,1])$ is a class of  vector functions with the values from $H_\infty \equiv\bigcap^\infty_{j=1}{D(A^j)}$  and finite in the vicinity of zero. By integrating by parts  it might be easily verified that  ${L^{'}}_0 $ is symmetric in $H_1$ . Denote its closure by $L_0$ and call it a minimal operator. Adjoint of $L_0$ is denoted by $L^\ast_0$ and  called a maximal operator.

In [1]  the following boundary value problem was considered:

\begin{equation}
l[y]\equiv y^{IV}(t) + Ay(t) = \lambda y(t)
\end{equation}
\begin{equation}
y^{'''} (0) = \lambda Q_{1}y(0), -y^{''}(0) = \lambda Q_{2}y^{'}(0) \end{equation}
\begin{equation}
 \cos C Y^{'}_b - \sin CY_{b}= 0 \end{equation}
where $Q_1{,Q}_2$ are defined as given  above, $Y_{b}=\left(y_b,y^{'}_b\right), Y^{'}_b  =(y^{'''}_b, y^{''}_b),$ and $y_b,y^{'}_b,y^{'''}_b, y^{''}_b$  are regularized values at $t=b$ of $y(t)$ and its derivatives to third order according to [2]. There the  questions of selfadjointness and compactness of the operator corresponding to that problem  in a larger space are studied. But we have the following  notes regarding statements in the indicated work:

1) The author  defines a minimal symmetric operator associated with problem (1.3)-(1.5) in space $L_{2} (H, (0, b))$ as a closure of the symmetric operator $L^{'}_0$  with domain   $D\left(L^{'}_0\right)=C_0^{\infty}(H_\infty, [0, b )),$ which is a set of infinitely many times differentiable vector functions with the  values  from $ H_\infty$,  finite in the vicinity of $b$ and  $L^{'}_0 y(t) \equiv ly$. It is stated in the 
work that the domain of closure of $L^{'}_0$ is given by (1.5), which  obviously is not true, since the closure in $L_{2} (H,(0, b))$ are the functions satisfying $y(b)=y^{'}(b)=y^{''}(b)=y^{'''}(b)=0$ which are just a part of the set of functions satisfying (1.5).

The adjoint operator is denoted by $L^{\ast}_0$.

2) In Theorem 1 from [1] it is stated that if  $Q_1=Q_2$ then the closure  of the operator $L^{'}_B$ with domain of definition  $\left\{Y=\left\{y\left(t\right),\ y_1,y_2\right\}\in H_1,\ y(t)\in C^\infty(H_\infty\left(0,b\right),\ \right\},$  where $y_1=Q_1y\left(0\right), y_2=Q_2y^{'}\left(0\right)$ and 
 $L^{'}_B=\left\{L^{\ast}_0 y\left(t\right), y^{'''}\left(0\right),-y^{''}(0)\right\}$ gives operator whose domain class of functions satisfying conditions (1.5) and which is selfadjoint. But this is obviously not true, since the  indicated closure consists of the vectors $Y=\left\{y\left(t\right),y_1,y_2\right\}$, $y\left(t\right)\in W^4_2\left(H,\left(0,b\right)\right), ly\in L_2\left(H,\left(0,b\right)\right)$ and  $y(t)\in D(A)$.

 For that reason,  we decide don't refer to that but give definitions of the minimal and  maximal operators and selfadjoint extensions, then treat some spectral questions for operator generated by $l[y]$ in direct sum space. Thus:

1. Define a minimal symmetric operator corresponding to differential expression (1.1)  with exit to direct sum space and give boundary conditions defining selfadjoint extensions of that operator.

2. Give conditions for that extension to be discrete or  to have spectrum filling some interval from the real axis. Also, define selfadjoint extensions whose resolvents are from ${\sigma }_p$ which is  Schatten von Neumann class of functions. For that, we will  follow a succession of steps  similar to steps [3], where the  domains of minimal and maximal operators are studied. Note here also [4,5], where selfadjoint extensions and eigenvalue asymptotics for Sturm-Liouville operator equation by exiting to  a larger space  and  [2], where selfadjoint extensions of operators generated by  $2n$-th order differential operator expressions (ones having unbounded operator coefficients) without exit to a  larger space are studied.

 3.Consider the eigenvalue problem
\begin{equation} \label{GrindEQ__6_}
ly=\lambda y
\end{equation}
\begin{equation} \label{GrindEQ__7_}
y(0)=y''(0)=0
\end{equation}
\begin{equation} \label{GrindEQ__8_}
-y'''(1)=\lambda Q_{1} y(1),\, \, \, y''(1)=\lambda Q_{2} y'(1).
\end{equation}
The operator  corresponding to this problem is one of selfadjoint  extensions of the  minimal operator corresponding to (1.1) with an exit to a larger space. We study its eigenvalue distribution.

 4. Give a trace formula for an operator associated with (1.6)-(1.8). For traces, a more general method than one used in our previous works will be suggested. The suggested method will let to treat these problems from unique  point of view.

Results for such problems are applicable to boundary value problems for some classes of partial differential equations .

Recall that since $q(t)$ is bounded  in $L_{2} \left(H,\left(0,1\right)\right)$, the existence of $q(t)$ in $l[y]$ is not essential for  domain of definitions of  minimal, maximal operators and self-adjoint extensions, that is why  when studying  these questions we will take $q(t)\equiv 0$.

We will use the following   notations :$H_{j} $ (a scale of Hilbert spaces generated by $A$)as always denotes $\left(j>0\right)$ a completion of $D(A^j)$ with respect to the scalar product $\left(f,g\right)_{j} =\left(A^{j} f,A^{j} g\right)$  (see [3]) for $j>k,\, \, H_{j} \subseteq H_{k} \subseteq H.$ $H_{-j} $  is a space with a negative norm constructed with respect to $H,\; H_{j} $. $H_{-j} $ is the completion of $H$ in the norm $\left\|A^{-j}f\right\|.$  $H_{-j}$ is usually considered as an adjoint to $H_{j} $ with respect to the scalar product $\left(•,•\right)$, so that for $g\in H_{-j} ,\, f\in H_{j} $, $g(f)$ will be written  as $\left(f,g\right)$. The operator $A$ is an isometric operator from $H_{1} $ to $H$. The adjoint of $A$ denoted  by $\tilde{A}$ acts from $H$ to $H_{-1} $ and is the extension of $A$.

\section {Domains of definition  of adjoint operator, selfadjoint extensions, selfadjoint extensions with compact resolvents }

Recall that  $q(t)\equiv 0$.

\begin{theorem}
The domain $D(L^*_0)$ of $L_0$ consists of those elements $Y=\left(y\left(t\right),Q_1y\left(1\right), \right.$\;  $\left.Q_2y^{'}(1)\right)$ of space $H_1=L_2\left(H,\left(0,1\right)\right)\oplus H^{2}_{Q}$, where
 \[
y\left( t\right) =e^{{\alpha }_{1}\sqrt[4]{\tilde{A}}t}f_{1}+e^{{\alpha }_{2}%
\sqrt[4]{\tilde{A}}t}f_{2}+e^{-\alpha _{1}\sqrt[4]{A}(t-1)}g_{1}+
\]%
\begin{equation}
+e^{-\alpha _{2}\sqrt[4]{A}(t-1)}g_{2}+\int_{0}^{1}{G\left( t,s\right) h(s)ds
},
\end{equation}
\begin{equation} 
G\left(t,s\right)=\left[\frac{e^{{\alpha }_1\sqrt[4]{A}\left|t-s\right|}}{{4\alpha }^3_1}+\frac{e^{{\alpha }_2\sqrt[4]{A}\left|t-s\right|}}{{4\alpha }^3_2}\right]A^{-\frac{3}{4}},
\end{equation}
\[
f_1,\ f_2\epsilon H_{-\frac{1}{8}}, g_1,\ g_2\epsilon H_{\frac{3}{4}}\; (\mbox{or}\; A^{\frac{3}{4}}g_i\in H,\ i=1,2),
\] 
\begin{equation}
Q_1 g_i\in H,\ A^{\frac{1}{4}}Q_{2\ }g_i\epsilon H,\; \mbox{for}\; i=1,2, \end{equation}
${\alpha }_1,{\alpha }_2$  are the roots of the equation ${\alpha }^4=-1$ with negative real parts, so, ${\alpha }_1=e^{\frac{3\pi i}{4}}, {\alpha }_2=e^{\frac{5\pi i}{4}}$
 and
\begin{equation} 
L^*_0Y=\left(ly,-y^{'''}\left(1\right),y^{''}(1)\right).
\end{equation}
\end{theorem}
 Since $g_1,\ g_2\epsilon H$ and $f(A)g=f\left(\tilde{A}\right)g$ for a bounded function $f$ on $H$, then in (2.1) in third and fourth terms we take $A$ but not $\tilde{A}$.
\begin{proof}
In [2],  $ly={\left(-1\right)}^ny^{(2n)}+Ay$   in $L_2\left(H,\left(0,1\right)\right)$ is considered, and there  it was shown that the values of $y(t)$ at endpoints of the interval are from a larger space than $H$, namely,  $f_1,\ f_2,\ g_1,\ g_2\epsilon H_{-\frac{1}{8}}$. But we take $f_1,\ f_2\epsilon H_{-\frac{1}{8}}$, and $g_1,\ g_2\epsilon H_{\frac{3}{4}}$ ,$Q_1g_i\in H,\ A^{\frac{1}{4}}Q_{2 }g_i\epsilon H,$ for $i=1,2$, because we define an operator in $H_1=L_2\left(H,\left(0,1\right)\right)\oplus H^2_Q$ and for that reason   $y^{'}(1),  y^{'''}(1), Q_1y\left(1\right),\ Q_2y^{''}(1)$ must be from $H$.
\end{proof}

As it follows from Theorem 2.1  (relations (2.1),(2.3)) values of $y(t)$  at zero are distributions.

Let $Y_0$ and $Y^{'}_0$ be defined in $H^2$  by
\begin{equation} 
Y_0=\left\{y_0,y^{'}_0\right\},\ \ Y^{'}_0=\{y^{'''}_0,y^{''}_0\}
\end{equation}
where $y_0,y^{'}_0, y^{'''}_0,y^{''}_0$ are regularized values of $y(t)$ and its derivatives at zero which are obtained from [2],by taking $n=2$  
\[y_0=A^{-\frac{1}{8}}y\left(0\right),\  y^{'}_0=A^{-\frac{3}{8}}y^{'}(0),\ {y}^{''}_0=A^{\frac{3}{8}}\left(y^{''}(0)-\sqrt{2}A^{\frac{1}{4}}y'(0)+A^{\frac{1}{2}}y(0)\right)\]
\begin{equation} 
y^{'''}_0=A^{\frac{1}{8}}\left(-y^{'''}\left(0\right)+A^{\frac{1}{2}}y^{'}\left(0\right)+\sqrt{2}A^{\frac{3}{4}}y(0)\right)
\end{equation}
Now let $\tilde{y}=\left\{y_0,{y'}_0\right\},\ \ \widetilde{y'}=\left\{{y'''}_0,{y''}_0\right\} $ be arbitrary vectors from $H^2$. By the similar way  done in [4], [3] the folowing lemma might be easily proved

\textbf{Lemma 2.1}
\textit{ For each} $\left\{\tilde{y},\ \tilde{y}'\right\}\epsilon H^4$ \textit{there exists } $Y=\left\{y\left(t\right),Q_1y\left(1\right),Q_2y^{'}(1)\right\}\in D( L^{\ast}_0)$ \textit{so that} $y_0,y^{'}_0,y^{''}_0,y^{'''}_0$ \textit{are defined by (2.6)}

 \textit{By the methods of the work [6] (where  condition for binary relations to be hermitian is given), [3],[2] and [4](where  Sturm-Liouville operator with unbounded operator coefficient and  with exit to larger space is defined) the following theorem might be easily verified.}

\begin{theorem} The domain of self-adjoint extensions $L^s_0$ of operator $L_0$ in $H_1$ consists of those $Y\in D\left(L^{\ast}_0\right)$ which  satisfy also
\begin{equation} 
\cos CY^{'}_0-\sin CY_0=0
\end{equation}
with a selfadjoint operator $C$ on $H^2 :C=\left(C_1,C_2\right),$  $C_i$ act in $H$ for $i=1,2$, and $Y_0,Y^{'}_0$ are defined by (2.5), (2.6). Without lost of generalization  for  simplifying notations we will take $C=\left(C,C\right)$.
\end{theorem}
 
 \textbf{Note 2.1.} Since $q(t)$ is
a selfadjoint and  bounded operator in $H_{1}$ the statement of the theorem remains true also for ${L=L}_0+Q$, where $Q=\left\{q\left(t\right),\ 0,0\right\}$

 Denote selfadjoint extension of $L$ by $L_s.$

\begin{theorem}
A spectrum of selfadloint extensions $L^s_0$ of minimal operator $L_0$ is discrete if and only if $\cos C, { Q}_1A^{-\frac{3}{4}},{Q}_2A^{-\frac{1}{2}}$ are compact.
\end{theorem}
\begin{proof}
Let $\lambda $ be non-real, then for selfadjoint extension $L^s_0$ and  $\tilde{h}=\left(h\left(t\right),\ h_1,h_2\right)\in H_1,$  $Y\in D\left(L^s_0\right)$, we consider the equation ,
\begin{equation} 
L^s_0Y-\lambda Y=\tilde{h}
\end{equation}
or in the equivalent form
\begin{equation} 
y^{IV}+Ay-\lambda y=h\left(t\right),
\end{equation}
\begin{equation} 
-y^{'''}\left(1\right)-\lambda Q_1y\left(1\right)=h_1,
\end{equation}
\begin{equation} 
y^{''}\left(1\right)-\lambda Q_2y'\left(1\right)=h_2
\end{equation}
moreover, $Y$ as a vector from the domain of $L^s_0$ satisfies the condition (2.7). From (2.8)-(2.11) the resolvent $R_{\lambda }(L^s_0)$ of $L^s_0$ is
\begin{equation} 
R_{\lambda }\left(L^s_0\right)\tilde{h}=Y=\left( \begin{array}{c}
y(t,\lambda ) \\
Q_1y(1) \\
Q_2y'(1) \end{array}
\right)
\end{equation}
where $y(t, \lambda )$ is the solution of (2.9) defined by
 \[
y\left(t,\lambda \right)=e^{{\alpha }_1\sqrt[4]{A-\lambda I}t}A^{\frac{1}{8}}f_1+e^{{\alpha }_2\sqrt[4]{A-\lambda I}t}{A^{\frac{1}{8}}f}_2+e^{{-\alpha }_1\sqrt[4]{A-\lambda I}\ (t-1)}A^{-\frac{3}{4}}g_1+\]
\begin{equation}
+e^{{-\alpha }_2\sqrt[4]{A-\lambda I }(t-1)}{A^{-\frac{3}{4}}g}_2+\int^1_0{G\left(t,s,\lambda \right)h(s)ds},
\end{equation}
where
\begin{equation} 
f_1,\ f_2\epsilon H, g_1,\ g_2\epsilon H,  Q_1A^{-\frac{3}{4}}g_i\in H,\ Q_{2\ }A^{-\frac{1}{2}}g_i\epsilon H,\; \mbox{for}\; i=1,2.
\end{equation}
Introduce the notations:
\begin{equation} 
{\omega }_j(t,\lambda )=\left\{ \begin{array}{c}
e^{{\alpha }_i\sqrt[4]{A-\lambda I}t}A^{\frac{1}{8}},\; i=1,2,\; j=1,2\qquad \qquad \\
e^{{-\alpha }_i\sqrt[4]{A-\lambda I}  (t-1)}{A^{-\frac{3}{4}}}, \;i=1,2 ,\; j=3,4 \end{array}
\right.
\end{equation}
where, ${\omega }_j\left(t,\lambda \right)f_i,(j=1,2,\ i=1,2)$ and ${\omega }_j\left(t,\lambda \right)g_i$ $(j=3,4, i=1,2)$ from a  fundamental system of solutions of the homogenous equation  corresponding to (2.9)

Let also $Y_0=\left\{y_0,y^{'}_0\right\},\ \ Y^{'}_0=\{y^{'''}_0,y^{''}_0\}$ whose elements are  defined by  (2.6). With (2.13)-(2.15) in mind we can write:
\[R_{\lambda }\left(L^s_0\right)\tilde{h}=\left( \begin{array}{c}
y(t) \\
Q_1y(1) \\
Q_2y'(1) \end{array}
\right)=
\]
\begin{small}
\[
=\left( \begin{array}{c}
{\omega }_1\left(t,\lambda \right)f_1+{\omega }_2{(t,\lambda )f}_2+{\omega }_3{\left(t,\lambda \right)g}_1+{\omega }_4{\left(t,\lambda \right)g}_2+\int^1_0{G\left(t,s,\lambda \right)h(s)ds} \\
{Q_1\omega }_1\left(1,\lambda \right)f_1+Q_1{\omega }_2{\left(1,\lambda \right)f}_2+Q_1A^{-\frac{3}{4}}g_1+{Q_1A^{-\frac{3}{4}}g}_2+Q_1\int^1_0{G\left(1,s,\lambda \right)h(s)ds} \\
{Q_2k_1\omega }_1\left(1,\lambda \right)f_1+Q_2{k_2\omega }_2\left(1,\lambda \right)f_2-Q_2{k_1A}^{-\frac{3}{4}}g_1-{Q_2{k_2A}^{-\frac{3}{4}}g_2}+Q_2\int^1_0{G_t^{'}\left(1,s,\lambda \right)h(s)ds} \end{array}
\right)\]
\end{small}
\[\ k_i={\alpha }_i\sqrt[4]{A-\lambda I}\ ,i=1,2.\]
Rewrite the last relation in the following  matrix  form:
\[
R_{\lambda }\left(L^s_0\right)\tilde{h}=\left( \begin{array}{c}
{\omega }_1\left(t,\lambda \right) \\
{Q_1\omega }_1\left(1,\lambda \right) \\
{Q_2k_1\omega }_1\left(1,\lambda \right) \end{array}
 \begin{array}{c}
{\omega }_2\left(t,\lambda \right) \\
Q_1{\omega }_2{\left(1,\lambda \right)} \\
Q_2{k_2\omega }_2\left(1,\lambda \right) \end{array}
 \begin{array}{c}
{\omega }_3\left(t,\lambda \right) \\
Q_1A^{-\frac{3}{4}} \\
-Q_2{k_1A}^{-\frac{3}{4}} \end{array}
 \begin{array}{c}
{\omega }_4\left(t,\lambda \right) \\
Q_1A^{-\frac{3}{4}} \\
{-Q}_2{k_2A}^{-\frac{3}{4}} \end{array}
\right)\left( \begin{array}{c}
f_1 \\
 \begin{array}{c}
f_2 \\
 \begin{array}{c}
g_1 \\
g_2 \end{array}
 \end{array}
 \end{array}
\right)
\]
\begin{equation} 
+\left( \begin{array}{c}
\int^1_0{G\left(t,s,\lambda \right),h(s)ds} \\
Q_1\int^1_0{G\left(1,s,\lambda\right)h(s)ds} \\
Q_2\int^1_0{G_t^{'}\left(1,s,\lambda\right)h(s)ds} \end{array}
\right)
\end{equation}
Define the vector $\left( \begin{array}{c}
f_1 \\
 \begin{array}{c}
f_2 \\
 \begin{array}{c}
g_1 \\
g_2 \end{array}
 \end{array}
 \end{array}
\right)$ from equalities  (2.7),(2.10),(2.11)  by substituting $y(t,\ \lambda )$ from (2.13) into them.  Introduce some notations.

 Firstly,  let
\[B_j\left(t,\lambda \right)=\left\{ \begin{array}{c}
\frac{e^{{-\alpha }_i\sqrt[4]{A-\lambda I}t}}{{4\alpha }^3_i}A^{-\frac{3}{4}},i=1,2,\ j=1,2 \\
\frac{e^{{\alpha }_i\sqrt[4]{A-\lambda I}t}}{{4\alpha }^3_i}A^{-\frac{3}{4}},i=1,2,\ j=3,4 \end{array}
\right. .\]
Note that with  this  notation
\begin{equation} 
G\left(t,s,\lambda \right)=\left\{ \begin{array}{c}
\left[{B_3(t,\lambda )e}^{{\alpha }_1\sqrt[4]{A-\lambda I}(-s)}+{B_4(t,\lambda )e}^{{\alpha }_2\sqrt[4]{A-\lambda I}(-s)}\right]\ ,\ \ s\le t \\
\left[{B_1(t,\lambda )e}^{{\alpha }_1\sqrt[4]{A-\lambda I}s}+B_2{(t,\lambda )e}^{{\alpha }_2\sqrt[4]{A-\lambda I}s}\right],\ \ \ t\le s \end{array}
\right.
\end{equation}
Now we will write boundary conditions (2.7),(2.10),(2.11) in the matrix form with $y\left(t\right)$ defined from (2.13) and with (2.17) in mind. For simplifying $G^{\left(n\right)}(0,s,\lambda ), n=1,4,$ (since $t=0\le s$ we have to use the second row expressions from (2.17)). Introduce the following notations,  obtained by taking $B_j(t,\lambda )$ in (2.6) instead of $y\left(t \right)$:
\[B^{'''}_{1,0}=A^{\frac{1}{8}}\left[-\frac{1}{{4\alpha }^3_1}{\left({\alpha }_1\sqrt[4]{A-\lambda I}\right)}^3A^{-\frac{3}{4}}+{\frac{1}{4{\alpha }^2_1}}\sqrt[4]{A-\lambda I}A^{-\frac{1}{4}}+\frac{\sqrt{2}}{{4\alpha }^3_1}\right]\]
\[B^{'''}_{2,0}=A^{\frac{1}{8}}\left[-\frac{1}{{4\alpha }^3_2}{\left({\alpha }_2\sqrt[4]{A-\lambda I}\right)}^3A^{-\frac{3}{4}}+{\frac{1}{{4\alpha }^2_2}}\sqrt[4]{A-\lambda I}A^{-\frac{1}{4}}+\frac{\sqrt{2}}{4{\alpha }^3_2}\right]\]
                              
\[B^{''}_{1,0}=A^{\frac{3}{8}}\left[\frac{1}{4{\alpha }^3_1}{\left({\alpha }_1\sqrt[4]{A-\lambda I}\right)}^2A^{-\frac{3}{4}}-{\frac{\sqrt{2}}{4{\alpha }^2_1}}\sqrt[4]{A-\lambda I}A^{-\frac{1}{2}}+\frac{\sqrt{2}}{{4\alpha }^3_1}A^{-\frac{1}{4}}\right]\]
\[B^{''}_{2,0}=A^{\frac{3}{8}}\left[\frac{1}{{4\alpha }^3_2}{\left({\alpha }_2\sqrt[4]{A-\lambda I}\right)}^2A^{-\frac{3}{4}}-{\frac{\sqrt{2}}{{4\alpha }^2_2}}\sqrt[4]{A-\lambda I}A^{-\frac{1}{2}}+\frac{\sqrt{2}}{{4\alpha }^3_2}A^{-\frac{1}{4}}\right]
\]
\[
B_{1,0}={\frac{1}{4{\alpha }^3_1}A}^{-\frac{1}{8}}A^{-\frac{3}{4}} ,      B_{2,0}{=\frac{1}{4{\alpha }^3_2}A}^{-\frac{1}{8}}A^{-\frac{3}{4}},       B^{'}_{1,0}={-A}^{-\frac{3}{8}}\frac{{\alpha }_1\sqrt[4]{A-\lambda I}}{{4\alpha }^3_1}A^{-\frac{3}{4}}\]
\begin{equation}
B^{'}_{2,0}=-A^{-\frac{3}{8}}\frac{{\alpha }_2\sqrt[4]{A-\lambda I}}{{4\alpha }^3_2}A^{-\frac{3}{4}}
\end{equation}

Regularized  values of  ${\omega }_j(t,\lambda $ ) and its derivatives at zero defined by (2.6) denote by ${\omega }_{j,0}$, ${\omega '}_{j,0}$, ${\omega ''}_{j,0}$, ${\omega '''}_{j,0}$ and  for a shortcut of notations denote the  values of ${\omega }_j(t,\lambda $)  and its derivatives at 1  by  ${\omega }_j\left(1\right),{\omega }_j'\left(1\right),{\omega }_j''\left(1\right),\ \ {\omega }_j'''\left(1\right)$, respectively. Hence,
\[{\omega }_j(1,\lambda )\equiv {\omega }_j\left(1\right)=\left\{ \begin{array}{c}
e^{{\alpha }_i\sqrt[4]{A-\lambda I}}A^{\frac{1}{8}},\;i=1,2,\; j=1,2 \\
{A^{-\frac{3}{4}}},\;i=1,2,\; j=3,4 \end{array}
\right., \]
\[{\omega }_j^{'}(1,\lambda )\equiv {\omega^{'}}_j\left(1\right)=\left\{ \begin{array}{c}
{{\alpha }_i\sqrt[4]{A-\lambda I}e}^{{\alpha }_i\sqrt[4]{A-\lambda I}}A^{\frac{1}{8}},\; i=1,2,\; j=1,2 \\
-{\alpha }_i\sqrt[4]{A-\lambda I}\ {A^{-\frac{3}{4}}},\; i=1,2 ,\; j=3,4 \end{array}
\right.\]
\[{\omega }_j^{''}(1,\lambda )\equiv {\omega }_j^{''}\left(1\right)=\left\{ \begin{array}{c}
{{\left({\alpha }_i\sqrt[4]{A-\lambda I}\right)}^2e}^{{\alpha }_i\sqrt[4]{A-\lambda I}}A^{\frac{1}{8}},\ \ \ i=1,2,\ j=1,2 \\
{\left({\alpha }_i\sqrt[4]{A-\lambda I}\right)}^2{A^{-\frac{3}{4}}},\; i=1,2,\; j=3,4 \end{array}
\right.\]
\[{\omega }_j'''(1,\lambda )\equiv {\omega }_j'''\left(1\right)=\left\{ \begin{array}{c}
{{\left({\alpha }_i\sqrt[4]{A-\lambda I}\right)}^3e}^{{\alpha }_i\sqrt[4]{A-\lambda I}}A^{\frac{1}{8}},\;i=1,2,\; j=1,2 \\
{\left(-{\alpha }_i\sqrt[4]{A-\lambda I}\right)}^3{A^{-\frac{3}{4}}},\; i=1,2,\; j=3,4 \end{array}
\right.\]

With all that notations substituting  $y(t, \lambda )$ into (2.7),(2.10),(2.11)  and then writing them in the matrix form  we have :

\[\left[\left( \begin{array}{c}
 \begin{array}{cc}
cosC &  \begin{array}{ccc}
O & O & O \end{array}
 \end{array}
 \\
 \begin{array}{ccc}
O & cosC &  \begin{array}{cc}
O & O \end{array}
 \end{array}
 \\
 \begin{array}{c}
 \begin{array}{ccc}
O & O &  \begin{array}{cc}
I & O \end{array}
 \end{array}
 \\
 \begin{array}{ccc}
O & O &  \begin{array}{cc}
O & I \end{array}
 \end{array}
 \end{array}
 \end{array}
\right)\left( \begin{array}{ccc}
{\omega '''}_{1,0} & {\omega '''}_{2,0} &  \begin{array}{cc}
{\omega '''}_{3,0} & {\omega '''}_{4,0} \end{array}
 \\
{\omega ''}_{1,0} & {\omega ''}_{2,0} &  \begin{array}{cc}
{\omega ''}_{3,0} & {\omega ''}_{4,0} \end{array}
 \\
 \begin{array}{c}
{\omega }_1'''\left(1\right) \\
{\omega }_1''\left(1\right) \end{array}
 &  \begin{array}{c}
{\omega }_2'''\left(1\right) \\
{\omega }_2''\left(1\right) \end{array}
 &  \begin{array}{cc}
 \begin{array}{c}
{\omega }_3^{'''}\left(1\right) \\
{\omega }_3^{''}\left(1\right) \end{array}
 &  \begin{array}{c}
{\omega }_4^{'''}\left(1\right) \\
{\omega }_4^{''}\left(1\right) \end{array}
 \end{array}
 \end{array}
\right)-\right.
\]
\[
-\left( \begin{array}{ccc}
\sin C & O &  \begin{array}{cc}
O & O \end{array}
 \\
O & sinC &  \begin{array}{cc}
O & O \end{array}
 \\
 \begin{array}{c}
O \\
O \end{array}
 &  \begin{array}{c}
O \\
O \end{array}
 &  \begin{array}{cc}
 \begin{array}{c}
I \\
O \end{array}
 &  \begin{array}{c}
O \\
I \end{array}
 \end{array}
 \end{array}
\right)\times
\]
\[
\left.\times\left( \begin{array}{ccc}
{\omega }_{1,0} & {\omega }_{2,0} &  \begin{array}{cc}
{\omega }_{3,0} & {\omega }_{4,0} \end{array}
 \\
{\omega '}_{1,0} & {\omega '}_{2,0} &  \begin{array}{cc}
{\omega '}_{3,0} & {\omega '}_{4,0} \end{array}
 \\
 \begin{array}{c}
{\lambda Q_1\omega }_1\left(1\right) \\
{\lambda Q_2\omega }_1'\left(1\right) \end{array}
 &  \begin{array}{c}
{\lambda Q_1\omega }_2\left(1\right) \\
{\lambda Q_2\omega }_2'\left(1\right) \end{array}
 &  \begin{array}{cc}
 \begin{array}{c}
{\lambda Q_1\omega }_3\left(1\right) \\
{\lambda Q_2\omega }_3'\left(1\right) \end{array}
 &  \begin{array}{c}
{\lambda Q_1\omega }_4\left(1\right) \\
{\lambda Q_2\omega }_4'\left(1\right) \end{array}
 \end{array}
 \end{array}
\right) \right]
\left( \begin{array}{c}
f_1 \\
f_2 \\
 \begin{array}{c}
g_1 \\
g_2 \end{array}
 \end{array}
\right)=
\]
\[
=\left[ \left(
\begin{array}{ccc}
cosCA^{-\frac{1}{8}}{B^{\prime \prime \prime }}_{1,0} & cosCA^{-\frac{1}{8}}{%
B^{\prime \prime \prime }}_{2,0}\  & O \\
cosC{A^{-\frac{1}{8}}B^{\prime \prime }}_{1,0} & cosC{A^{-\frac{1}{8}%
}B^{\prime \prime }}_{2,0} & O \\
O & O & \frac{A^{-\frac{7}{8}}}{4{\alpha }_{1}^{3}}\left[ {\left( {\alpha }%
_{1}\sqrt[4]{A-\lambda I}\right) }^{3}+\lambda Q_{1}\right]  \\
O & O & \frac{A^{-\frac{7}{8}}}{4{\alpha }_{1}^{3}}\left[ {\left( {\alpha }%
_{1}\sqrt[4]{A-\lambda I}\right) }^{2}+\lambda Q_{2}{\alpha }_{1}\sqrt[4]{%
A-\lambda I}\right]
\end{array}%
\right. \right.
\]%

\[
\left. \left.
\begin{array}{ccc}
O & O & O \\
O & O & O \\
\frac{A^{-\frac{7}{8}}}{4{\alpha }_{2}^{3}}\left[ {\left( {\alpha }_{2}\sqrt[%
4]{A-\lambda I}\right) }^{3}+\lambda Q_{1}\right]  & I & O \\
\frac{A^{-\frac{7}{8}}}{4{\alpha }_{2}^{3}}\left[ {\left( {\alpha }_{2}\sqrt[%
4]{A-\lambda I}\right) }^{2}+\lambda Q_{2}\left( {\alpha }_{2}\sqrt[4]{%
A-\lambda I}\right) \right]  & O & I%
\end{array}%
\right) \right. -
\]
\[
-\left.\left(
\begin{array}{cccccc}
sinCA^{-\frac{1}{8}}B_{1,0} & sinCA^{-\frac{1}{8}}B_{2,0}\  & O & O & O & O
\\
sinCA^{-\frac{1}{8}}{B^{\prime }}_{1,0} & sinCA^{-\frac{1}{8}}{B^{\prime }}%
_{2,0} & O & O & O & O \\
O & O & O & O & O & O \\
O & O & O & O & O & O%
\end{array}%
\right) \right]\times
\]

\begin{equation} 
\times \left(
\begin{array}{c}
A^{\frac{1}{8}}\int_{0}^{1}{e^{{\alpha }_{1}\sqrt[4]{A-\lambda I}s}h(s)ds}
\\
A^{\frac{1}{8}}\int_{0}^{1}{e^{{\alpha }_{2}\sqrt[4]{A-\lambda I}s}h(s)ds}
\\
A^{\frac{1}{8}}\int_{0}^{1}{e^{{\alpha }_{1}\sqrt[4]{A-\lambda I}\
(1-s)}h(s)ds} \\
A^{\frac{1}{8}}\int_{0}^{1}e^{{\alpha }_{2}\sqrt[4]{A-\lambda I}\ \left(
1-s\right) }h\left( s\right) ds \\
h_{1} \\
h_{2}%
\end{array}%
\right)
\end{equation}
In the last column in (2.19) the integral terms are from $D(A^{\frac{1}{8}})$ and to get terms from $H$ we put before  them the factors $A^{\frac{1}{8}}$,  that is why there appear the factors $A^{-\frac{1}{8}}$ in front of $B's$  in  matrices within braces in right of (2.19).

Denote the matrix within the brackets in the left hand side of (2.19) by $D,$ the difference of matrices  in the brackets on the right of (2.16) by $\widetilde{B}_1-\widetilde{B}_2$ and column matrix on the right by   $\widetilde{H},$ respectively. Each term of $\tilde{H}$ is from  $H$, because  linear operators $h(s)\to \int\limits^1_0{e^{{\alpha }_i\sqrt[4]{A-\lambda I}s}h(s)ds},\; h(s)\to \int^1_0{e^{{\alpha }_i\sqrt[4]{A-\lambda I} (1-s)}h(s)ds}$  continuously act from $L_2\left(H,\left(0,1\right)\right)$ to  $H_{\frac{1}{8}}$  as the adjoint to the operator $f\to e^{{\alpha }_i\sqrt[4]{A-\lambda I}s}f$ which is continuously acts from $H_{-\frac{1}{8}}$ to $L_2\left(H,\left(0,1\right)\right)$ (see [3]). Hence,

\begin{equation} 
\left( \begin{array}{c}
f_1 \\
f_2 \\
 \begin{array}{c}
g_1 \\
g_2 \end{array}
 \end{array}
\right)=D^{-1}\left(\widetilde{B}_1-\widetilde{B}_2\right)\tilde{H},
\end{equation}
Denote in (2.16), the matrix in front of the vector $\left( \begin{array}{c}
f_1 \\
 \begin{array}{c}
f_2 \\
 \begin{array}{c}
g_1 \\
g_2 \end{array}
 \end{array}
 \end{array}
\right)$ by $J$  and the second term by G

\begin{equation} 
R_{\lambda }\left(L^0_s\right)\tilde{h}=J\widetilde{D}^{-1} (B_1-\widetilde{B}_2)\widetilde{H}+G,
\end{equation}
In a similar way as in [7] it might be  easily verified that $D^{-1}$ is bounded  in $H$. In the matrix operator  $\widetilde{B_2}$ the term $\sin C  A^{-\frac{1}{8}}B_{i,0}$ is  $\sin C A^{-\frac{1}{8}}B_{i,0}=\frac{\sin CA^{-1}}{4\alpha_i ^{3}}$ and in $\widetilde{B}_1$ the term $\cos C A^{-\frac{1}{8}}B^{'''}_{i,0}$  (in other terms too) is representable as $\cos C A^{-\frac{1}{8}}B^{'''}_{i,0}=
\cos C(\beta I+F)$  where $\beta$ is a number defined by the  coefficients of the terms $\widetilde{B}_1$, $F$ is a bounded  operator in $H$.   It follows that $R_{\lambda }\left(L^{0}_{s}\right)$ is compact if and only if $\cos C, Q_1 A^{-\frac{3}{4}},Q_2 A^{-\frac{1}{2}}$ are  compact. 
\end{proof}

\textbf{Note 2.2.}  Since $q(t)$ is bounded the in $L_2\left(H,\left(0,1\right)\right)$, then  in virtue of relation
\begin{equation} 
R_{\lambda }\left(L_s\right)=R_{\lambda }\left(L^0_s\right)-R_{\lambda }\left(L_s\right)QR_{\lambda }\left(L^0_s\right),
\end{equation}
where  $L_s=L^s_0+Q,$
\begin{equation} 
Q Y=\left\{q\left(t\right)y\left(t\right),\ 0,0\right\} ,
\end{equation}
statement of Theorem 2.3 holds also for selfadjoint extensions $L_s$  of the minimal operator $L=L_0+Q$.

 \section{Asymptotics of eigenvalue distribution of one class of selfadjoint extensions and definition of the domain of selfadjoint extensions whose resolvents are from $\sigma_p$ and extensions with continuous spectrum }

 Take in boundary conditions (2.7)
 $C=\left(C_{1}, C_{2} \right)$, where $C_{1} $ and $C_{2} $  are operators on $H$, moreover $C_1=\frac{\pi }{2}I$ ($I$ is an identity operator in $H$), $C_2=arctg(-\sqrt{2}A)$, then the corresponding selfadjoint extension  will be given by the boundary conditions $y(0)=y^{''}(0)=0$. The eigenvalue problem corresponding to that operator  is:
\begin{equation} 
ly=\lambda y
\end{equation}
\begin{equation} 
y(0)=y^{''}(0)=0.
\end{equation}
\begin{equation} 
-y^{'''}(1)=\lambda Q_{1} y(1),\, \, \, y^{''}(1)=\lambda Q_{2} y^{'}(1)
\end{equation}
Note here that boundary conditions (3.2),(3.3) are obtained from (1.4),(1.5) with indicated above choice of $C$ and by the setting $b=1$ and making a change of variable  $1-t=x$.

The operator corresponding to that problem  for $q(t)\equiv 0$ denote by $L^0_1$ which in virtue of Theorems 2.2 and 2.3 is self-adjoint and discrete.

 Now we study the asymptotics of eigenvalues of that operator.

 The solution of (3.1)  is
\[
y\left(t,\lambda \right)=e^{{\alpha }_1\sqrt[4]{A-\lambda I}t}A^{\frac{1}{8}}f_1+e^{{\alpha }_2\sqrt[4]{A-\lambda I}t}{A^{\frac{1}{8}}f}_2+\]
\begin{equation} 
+e^{{-\alpha }_1\sqrt[4]{A-\lambda I }(t-1)}{A^{-\frac{3}{4}}g}_1+e^{{-\alpha }_2\sqrt[4]{A-\lambda I (}t-1)}{A^{-\frac{3}{4}}g}_2
\end{equation}
with  $f_1,\ f_2,g_1,g_2\in H$ defined as in (2.14).

Substituting  the   function (3.4) in the the boundary conditions (3.2), 
\begin{equation} 
A^{\frac{1}{8}}f_1+{A^{\frac{1}{8}}f}_2+e^{{\alpha }_1\sqrt[4]{A-\lambda I\ \ }}{A^{-\frac{3}{4}}g}_1+e^{{\alpha }_2\sqrt[4]{A-\lambda I}}{A^{-\frac{3}{4}}}g_2=0
\end{equation}
\[
{{\left({\alpha }_1\sqrt[4]{A-\lambda I}\right)}^2A}^{\frac{1}{8}}f_1+{{{\left({\alpha }_2\sqrt[4]{A-\lambda I}\right)}^2A}^{\frac{1}{8}}f}_2+{{\left({\alpha }_1\sqrt[4]{A-\lambda I }\right)}^2}e^{{\alpha }_1\sqrt[4]{A-\lambda I }}\times
\]
\begin{equation} 
\times{A^{-\frac{3}{4}}g}_1+{{\left({\alpha }_2\sqrt[4]{A-\lambda I }\right)}^2e}^{{\alpha }_2\sqrt[4]{A-\lambda I}}{A^{-\frac{3}{4}}g}_2=0
\end{equation}
rewrite (3.5),(3.6) as
\[
\left(A^{\frac{1}{8}}f_1+e^{{\alpha }_1\sqrt[4]{A-\lambda I\ \ }}{A^{-\frac{3}{4}}g}_1\right)+\left({A^{\frac{1}{8}}f}_2+e^{{\alpha }_2\sqrt[4]{A-\lambda I}}{A^{-\frac{3}{4}}g}_2\right)=0\]
\[{\left({\alpha }_1\sqrt[4]{A-\lambda I}\right)}^2\left(A^{\frac{1}{8}}f_1+e^{{\alpha }_1\sqrt[4]{A-\lambda I\ \ }}{A^{-\frac{3}{4}}g}_1\right)+
\]\[
+{\left({\alpha }_2\sqrt[4]{A-\lambda I}\right)}^2\left({A^{\frac{1}{8}}f}_2+e^{{\alpha }_2\sqrt[4]{A-\lambda I}}{A^{-\frac{3}{4}}g}_2\right)=0\]
hence
\begin{equation} \label{GrindEQ__2_7_}
f_1=-A^{-\frac{1}{8}}e^{{\alpha }_1\sqrt[4]{A-\lambda I\ }}{A^{-\frac{3}{4}}g}_1={-A}^{-\frac{7}{8}}e^{{\alpha }_1\sqrt[4]{A-\lambda I\ \ }}g_1={-A}^{-\frac{7}{8}}e^{\sqrt[4]{\lambda I-A}}g_1,
\end{equation}
 \[
f_2={-A}^{-\frac{1}{8}}e^{{\alpha }_2\sqrt[4]{A-\lambda I}}{A^{-\frac{3}{4}}g}_2=-A^{-\frac{7}{8}}e^{{\alpha }_2\sqrt[4]{A-\lambda I}}
g_2=\]
\begin{equation}
=-A^{-\frac{7}{8}}e^{i{\alpha }_1\sqrt[4]{A-\lambda I}}g_2=-A^{-\frac{7}{8}}e^{i\sqrt[4]{\lambda I-A}}g_2
\end{equation}
Taking in (3.4)  $f_1$ and $f_2\ $ as in (3.7), (3.8) yields
\[y\left(t,\lambda \right)=-e^{{\alpha }_1\sqrt[4]{A-\lambda I}t}A^{\frac{1}{8}}A^{-\frac{7}{8}}e^{\sqrt[4]{\lambda I-A }}g_1-e^{{\alpha }_2\sqrt[4]{A-\lambda I}t}A^{\frac{1}{8}}A^{-\frac{7}{8}}e^{i\sqrt[4]{\lambda I-A}}g_2+
\]\[
+e^{-\sqrt[4]{\lambda I-A (}t-1)}{A^{-\frac{3}{4}}g}_1+e^{-\sqrt[4]{\lambda I-A (}t-1)}{A^{-\frac{3}{4}}g}_2=-e^{\sqrt[4]{\lambda I-A}t}{A^{-\frac{3}{4}}e}^{\sqrt[4]{\lambda I-A }}g_1+
\]
\[+e^{-\sqrt[4]{\lambda I-A} (t-1)}{A^{-\frac{3}{4}}g}_1-e^{i\sqrt[4]{\lambda I-A}t}{A^{-\frac{3}{4}}e}^{i\sqrt[4]{\lambda I-A}}g_2+
\]\[
+e^{-i\sqrt[4]{\lambda I-A}(t-1)}{A^{-\frac{3}{4}}g}_2={-2sh\sqrt[4]{\lambda I-A}tA^{-\frac{3}{4}}e}^{\sqrt[4]{\lambda I-A\ \ }}g_1-2i{{\sin\sqrt[4]{\lambda I-A}tA}^{-\frac{3}{4}}e}^{i\sqrt[4]{\lambda I-A}}g_2.\]
Denoting $F_1={{-2e}^{\sqrt[4]{\lambda I-A\ }}{A^{-\frac{3}{4}}g}_1,\ \ \ F}_2={2ie}^{\sqrt[4]{\lambda I-A}}{A^{-\frac{3}{4}}g}_2$, $F_1,\ F_2\epsilon H$
 we have
\begin{equation} 
y\left(t\right)=sh\sqrt[4]{\lambda I-A}tF_1+{\sin\sqrt[4]{\lambda I-A}tF}_2
\end{equation}
Writing that solution in boundary conditions (3.3),  from expansion of a selfadjoint operator with discrete spectrum  $A=\sum\limits^{\infty }_{k=1}{{\gamma }_k}\left(\cdot ,{\varphi }_k\right){\varphi }_k$,   where   ${\gamma }_k$ are eigenvalues and ${\varphi }_k$ are  orthonormal basis formed by eigenvectors of $A,$ we have
\[{-\sqrt[4]{\lambda -{\gamma }_k}}^3ch\sqrt[4]{\lambda -{\gamma }_k}\left(F_1,\ {\varphi }_k\right)+{\sqrt[4]{\lambda -{\gamma }_k}}^3cos\sqrt[4]{\lambda -{\gamma }_k}\left(F_2,\ {\varphi }_k\right)=
\]
\begin{equation}
=\lambda sh\sqrt[4]{\lambda -{\gamma }_k}\left(F_1,\ {Q_1\varphi }_k\right)+\lambda sin\sqrt[4]{\lambda -{\gamma }_k}\left(F_2,\ {Q_1\varphi }_k\right)
\end{equation}
and
\[\sqrt{\lambda -{\gamma }_k}sh\sqrt[4]{\lambda -{\gamma }_k}\left(F_1,\ {\varphi }_k\right)-\sqrt{\lambda -{\gamma }_k}sin\sqrt[4]{\lambda -{\gamma }_k}\left(F_2,\ {\varphi }_k\right)=\]
\[\lambda \sqrt[4]{\lambda -{\gamma }_k}ch\sqrt[4]{\lambda -{\gamma }_k}\left(F_1,\ {Q_2\varphi }_k\right)+\lambda \sqrt[4]{\lambda -{\gamma }_k}cos\sqrt[4]{\lambda -{\gamma }_k}\left(F_2,\ {Q_2\varphi }_k\right).\]
 For simplicity of calculations take here
 $Q_1=Q_2=A^{\alpha },\; \mbox{and}\; 0<\alpha <\frac{1}{2}$, which is important for holding the hypothesis of theorem 2.3. 

By replacing
\begin{equation} 
\sqrt[4]{\lambda -{\gamma }_k}=z, \left(F_1,\ {\varphi }_k\right)=c_{1k},\ \left(F_2,\ {\varphi }_k\right)=c_{2k}
\end{equation}
in virtue of ${{Q_i{\varphi }_k=\gamma }_k}^{\alpha }{\varphi }_k,\; i=1,2$    we have
\begin{equation} 
-z^3chzc_{1k}+z^3\cos z{c}_{2k}=\left(z^4+{\gamma }_k\right)shz {\gamma }^{\alpha }_kc_{1k}+\left(z^4+{\gamma }_k\right)\sin z {\gamma }^{\alpha }_kc_{2k}
\end{equation}
\begin{equation} 
zshzc_{1k}-z\sin z{ c}_{2k}=\left(z^4+{\gamma }_k\right)chz {\gamma }^{\alpha }_kc_{1k}+\left(z^4+{\gamma }_k\right)\cos z {\gamma }^{\alpha }_kc_{2k}
\end{equation}
which is a system of equations in $c_{1k}$, $c_{2k}$ and has nonzero roots if and only if the characteristic determinant  $\Delta (z)$ of boundary value problem (3.1)-(3.3)  (determinant formed by coefficients of (3.12, (3.13)) is zero:
\begin{equation} 
\Delta \left(z\right)=\left| \begin{array}{cc}
-z^3chz-\left(z^4+{\gamma }_k\right)shz\ {\gamma }^{\alpha }_k & z^3 \cos z-\left(z^4+{\gamma }_k\right)\sin z\ {\gamma }^{\alpha }_k \\ \\
zshz-\left(z^4+{\gamma }_k\right)chz\ {\gamma }^{\alpha }_k & -zsinz-\left(z^4+{\gamma }_k\right)\cos z\ {\gamma }^{\alpha }_k \end{array}
\right|=0
\end{equation}
After simplifications in (3.14)
\begin{equation} 
tgz=\frac{-2z^3\left(z^4+{\gamma }_k\right)\ {\gamma }^{\alpha }_k+z^4thz-{\left(z^4+{\gamma }_k\right)}^2{\gamma }^{2\alpha }_kthz}{z^4+2z\left(z^4+{\gamma }_k\right)\ {\gamma }^{\alpha }_kthz-{\left(z^4+{\gamma }_k\right)}^2{\gamma }^{2\alpha }_k}
\end{equation}
Since $\lambda $ as an  eigenvalue of a selfadjoint operator must be real, feasible values for  $z$ are $y,-y, iy,-iy (y>0)$ or $\pm y\pm iy$.

Thus,  (3.15) might have only real, imaginary roots and roots of the form $y\pm iy, y$ is real. It can't have other complex roots with exception of these roots, because complex roots will  give complex eigenvalues for a selfadjoint operator, which is impossible.

 Hence, setting $\sqrt[4]{{\lambda }_{k,j}-{\gamma }_k}=z_{k,j}$ from (3.15) for real roots as $\left|z\right|\to \infty $ we have
 $z=z_{k,j}\sim \frac{\pi }{4}+\pi j+O\left(\frac{1}{k}\right), k$ is an entire number large in modulus  .

 Writing  in (3.15) iz in place of $z$ shows that if $z$ is a real root, then iz is also  a root of that equation, thus for imaginary roots
\[z=iz_{k,j}\sim \left(\frac{\pi }{4}+\pi j+O\left(\frac{1}{k}\right)\right)i,
\]
But in virtue of (3.11)
\[\lambda =z^4+{\gamma }_k\]
which shows that real and imaginary roots of  (3.15) result in the same eigenvalues but in linearly dependent  eigenvectors of the  operator and that is why algebraic multiplicity of each eigenvalue corresponding those roots is 2.

 Writing in (3.15)  $y\pm iy$ in place of $z$ after simplifications we get
\[-4y^4\left[\frac{sin2y}{2}-\frac{sh2y}{2i}\right]+4iy^2\left(y+iy\right)\left({\gamma }_k-4y^4\right){{\gamma }_k}^{\alpha }\left[\frac{cos2y}{2i}+\frac{ch2y}{2}\right]\]
\[+2y\left(y+iy\right)\left({\gamma }_k-4y^4\right){{\gamma }_k}^{\alpha }\left[\frac{cos2y}{2i}-\frac{ch2y}{2i}\right]+4y^4\left[\frac{sh2y}{2}+\frac{isin2y}{2}\right]+\]
\begin{equation} 
+{\left({\gamma }_k-4y^4\right)}^2{{\gamma }_k}^{2\alpha }\left[\frac{sh2y}{2}+\frac{isin2y}{2}\right]-{\left({\gamma }_k-4y^4\right)}^2{{\gamma }_k}^{2\alpha }\left[\frac{sin2y}{2}+\frac{sh2y}{2i}\right]=0
\end{equation}
It is easy to see that the right side of (3.16) has the form $K=iK$, where $K$ is real and
\[
K=-4y^4\frac{sin2y}{2}-4y^3\left({\gamma }_k-4y^4\right) {{\gamma }_k}^{\alpha }\left[\frac{cos2y}{2}+\frac{ch2y}{2}\right]+2\left({\gamma }_k-4y^4\right){{\gamma }_k}^{\alpha }y\left[\frac{cos2y}{2}-\frac{ch2y}{2}
\right]+
\]
\[+4y^4\frac{sh2y}{2}+{\left({\gamma }_k-4y^4\right)}^2{{\gamma }_k}^{2\alpha }\frac{sh2y}{2}-{\left({\gamma }_k-4y^4\right)}^2{{\gamma }_k}^{2\alpha }\frac{sin2y}{2}
\]
Equation (3.16) has roots if and only if
\begin{equation} 
K=0.
\end{equation}
But for large values of $y$ (3.17) is equivalent to
\[4y^7{{\gamma }_k}^{\alpha }e^{2y}+y^7{{\gamma }_k}^{\alpha }e^{2y}+8y^8{{\gamma }_k}^{2\alpha }e^{2y}+4y^7{{\gamma }_k}^{\alpha }e^{2y}+y^5{{\gamma }_k}^{\alpha }e^{2y}
\]
which has no positive roots. By Descarte's rule of signs (3.16)  has  small in modulus roots, but they don't change asymptotics of eigenvalues.

Hence, we get the next theorem
\begin{theorem} The algebraic multiplicity of  eigenvalues ${\lambda }_{k,j}$ of the operator $L^0_1$ is two and the following  asymptotic  formula is true:
\begin{equation} \label{GrindEQ__2_18_}
{\lambda }_{k,j}={\gamma }_k+z^4_{k,j},   z_{k,j}\sim \left\{ \begin{array}{c}
\pi j+\frac{\pi }{4}+O\left(\frac{1}{k}\right), \\
i\left(\pi j+\frac{\pi }{4}+O\left(\frac{1}{k}\right)\right) \end{array}
\right.
\end{equation}
and
${\lambda }_{k,j}={\gamma }_k+{\eta }^4_{k,j}$, where ${\eta }_{k,j}$ are  small in modulus roots  of (3.15).

Analogous to [5,8,9] the following statements  might be justified:
\end{theorem}

\textbf{Lemma 3.1.} \textit{For distribution function $N\left(\lambda \right)=\sum_{{\lambda }_{n<\lambda }}{1}$ of the eigenvalues of the operator $L^0_1$ the following relation is valid}

\begin{equation} 
N(\lambda )\sim C_1{\lambda }^{\frac{4+\alpha }{4\alpha }}
\end{equation}
\textit{for sufficiently large $\lambda $ .}

\begin{lemma}
For large  values of $n$ the following asymptotic formula is true
\begin{equation} 
{\lambda }_n\sim C_2n^{\frac{4\alpha }{4+\alpha }} , n\to \infty .
\end{equation}
Setting  $L_1=L^0_1+Q$,  one can easily see that hypotheses of
Theorem 2.2 and Theorem 2.3 hold also for $L_1\ $. Denote the eigenvalues of $L_1$ by $\left\{{\mu }_n\right\}:\ {\mu }_1<{\mu }_2<\dots \ \ .$
\end{lemma}
\textbf{Note 3.1} From (3.19) it follows that inverse of  $L^0_1$ is from  Neumann von Shcetten class ${\sigma }_p$, if and only if $p\cdot \frac{4\alpha }{4+\alpha }>1$ or $\alpha >\frac{4}{4p-1}$, which means that $A^{-\frac{1}{4}}\in {\sigma }_{4p-1}$.
\begin{lemma}
The operator ${\left(L^0_1\right)}^{-1}$  is from ${\sigma }_p$ if and only if  $A^{-\frac{1}{4}}\in {\sigma }_{4p-1}.$

Because of relation (2.22) that statement holds also for the operator $L_1$.
\end{lemma}
 Now we can prove the next theorem.

\textbf{Theorem 3.2} \textit{Let $A^{-\frac{1}{4}}\in {\sigma }_{4p-1}$. Then $R_{\lambda }\left(L^0_s\right),\ R_{\lambda }\left(L_s\right)\epsilon {\sigma }_p\left(H_1\right) $ if and only if \linebreak $cosC,\ {Q}_1A^{-\frac{3}{4}},{Q}_2A^{-\frac{1}{2}}$ are from ${\sigma }_p.$}

\begin{proof}
Setting in formula  (2.21)   $C=(C_1,C_2)$, $C_1=\frac{\pi }{2}I$, $C_2=arctg(-\sqrt{2}A)$ and  subtracting obtained by that way formula from (2.21)  we get
\begin{equation} 
R_{\lambda }\left(L^s_0\right) \tilde{h}=R_{\lambda }\left(L_0\right) \tilde{h}+\ JD^{-1}\left(\widetilde{B}_1-\widetilde{B}_2\right)\tilde{H}-J{D_0}^{-1}\left(\widetilde{B}_{01}-\widetilde{B}_{02}\right)\tilde{H}
\end{equation}
where $D_{0}$ and ${\widetilde{B}_{01}-\widetilde{B}_{02}}_0$ are obtained from $D$ and  $\widetilde{B}_1-\widetilde{B}_2$ by taking there $C=(C_1,C_2)$, $C_1=\frac{\pi }{2}I$, $C_2=arctg(-\sqrt{2}A)$.  Writing  (3.21) in the open form and with lemma 3.3  in mind one can  easily see that $R_{\lambda }\left(L^0_s\right)$ is  from ${\sigma }_p$ if and only if  $cosC,\ {Q}_1A^{-\frac{3}{4}},{\ \ Q}_2A^{-\frac{1}{2}}$ are from ${\sigma }_p.$
Since $Q$ is bounded in $H_1$, statement of theorem is true also for $R_\lambda (L_s)$.
\end{proof}
\textbf{Theorem 3.3.}
\textit{If the inverse of the operator $A$ is compact, then for any closed set $F$ of the real axis, there exists a selfadjoint extension of the minimal operator , whose spectrum coincides with  $F.$}

\begin{proof} Let $C=\left( \begin{array}{cc}
\frac{\pi }{2}I & O \\
O & f(A) \end{array}
\right)$, where $f(\mu )$ is any function Borel measurable on ($1,\infty $). Then boundary conditions (2.7) take on the form
\begin{equation} 
y\left(0\right)=0,
\end{equation}
\begin{equation} 
cosf\left(A\right)A^{\frac{3}{8}}\left(y^{''}(0)-\sqrt{2}A^{\frac{1}{4}}y'(0)+A^{\frac{1}{2}}y(0)\right)-sinf\left(A\right)A^{-\frac{3}{8}}y^{'}\left(0\right)=0
\end{equation}
Let us corresponding to $f$ selfadjoint extension  denote by $L_f.$ Obviously, $\lambda $ is an  eigenvalue of $L_f$ , if in addition to (3.22) and (3.23) holds (3.3). After  substituting $y(t,\lambda )$ from (3.4) into those relations and denoting by $K^f_{\lambda }(A)$ the determinant of the matrix formed by the coefficients of $f_1,\ f_2,\ g_1,\ g_2$ in relations (3.22), (3.23),(3.3), the further justifications are lead  as in Theorem 5 from [3].
\end{proof}

\section{ Orthonormal eigenvectors of operator $L^{0}_{1}$}

 Reminding that by $\left\{{\varphi }_{\kappa }\right\}$ we denote orthonormal eigenvectors of the operator  $A$,  then  orthogonal eigenvectors of $L^0_1$  will be as 
 \[
Y_{k,j}=\left\{c_{1k,j}sh\sqrt[4]{{\lambda }_{k,j}-{\gamma }_k}t{\varphi }_k+\ c_{2k,j}sin\sqrt[4]{{\lambda }_{k,j}-{\gamma }_k}t{\varphi }_{k },c_{1k,j}{{\gamma }_k}^{\alpha }sh\sqrt[4]{{\lambda }_{k,j}-{\gamma }_k}{\varphi }_k+\right.
\]\[
+c_{2k,j}\ {{\gamma }_k}^{\alpha }sin\sqrt[4]{{\lambda }_{k,j}-{\gamma }_k}{\varphi }_{k\ }, c_{1k,j}\sqrt[4]{{\lambda }_{k,j}-{\gamma }_k}{{\gamma }_k}^{\alpha }ch\sqrt[4]{{\lambda }_{k,j}-{\gamma }_k}{\varphi }_k +\]
\begin{equation}
  \left.+c_{2k,j}\ {{\sqrt[4]{{\lambda }_{k,j}-{\gamma }_k}\gamma }_k}^{\alpha }cos\sqrt[4]{{\lambda }_{k,j}-{\gamma }_k}{\varphi }_{k\ }\right\},\ k=\overline{1,\infty },\ j=\overline{1,\infty }.
\end{equation} 

Coefficients $c_{1k,j}$ and $c_{2k,j}$ are the values of $c_{1k}$ and $c_{2k}$ in relations (3.12), (3.13) obtained by taking  in (3.11) ${\lambda =\lambda }_{k,j}$. From (3.13)
\begin{equation} 
{\ c}_{1k}=\frac{zsinz+\left(z^4+{\gamma }_k\right)cosz\ {\gamma }^{\alpha }_k}{\ \ \ zshz-\left(z^4+{\gamma }_k\right)chz\ {\gamma }^{\alpha }_k }c_{2k}
\end{equation}
The multiplier at $c_{2k}$ we denote by $H(z):$
\begin{equation} 
c_{1k}=H(z)c_{2k}
\end{equation}
Thus, 
\begin{equation}
 c_{1k,j}=H(z_{k,j})c_{2k,j} \end{equation}
 Assuming for shortcut of notations $c_{2k,j}=c_{k,j}$  with all above in mind we have the following expressions for $Y_{k,j}:$
\[
Y_{k,j}=c_{k,j}\left\{\sin z_{k,j}t{\varphi }_{k\ }+H(z_{k,j})shz_{k,j}t{\varphi }_k\ ,\gamma _{k}^{\alpha }sinz_{k,j}{\varphi }_{k}+\right.
\]
\[
\left.+H(z_{k,j})\gamma_k^{\alpha }shz_{k,jk}\varphi_k,\; z_{k,j}\gamma_k^{\alpha}\cos z_{k,j}\varphi _{k}
+H(z_{k,j})z_{k,j}\gamma_k^{\alpha }chz_{k,j}\varphi_{k} \right \}=
\]
\[
=c_{k,j}{\varphi }_{k}\left\{\sin z_{k,j}t+H\left(z_{k,j}\right)sh z_{k,j}t,\gamma_{k}^{\alpha }\sin z_{k,j}+H(x_{k,j})\gamma_k^{\alpha}sh z_{k,j},  z_{k,j}\gamma_k^{\alpha }\cos z_{k,j}+\right.
\]
\begin{equation} 
\left.+H(z_{k,j})z_{k,j}\gamma _k^{\alpha }chz_{k,j}\ \right\},\ k=\overline{1,\infty },\ j=\overline{1,\infty }
\end{equation}
Introduce the  direct sum space $\Lambda =L_2\left(\left(0,1\right)\right)\oplus C^2$, ($C$ is a complex\ space ) with a scalar product of elements $u=\left(u\left(t\right),u_1,u_2\right),v=(v\left(t\right),v_1,v_2)$ defined as
\[\left(u,v\right)_{\Lambda}=\int\limits^{1}_{0} u\overline{v}dt+\gamma^{\alpha}_{k} u\left(1\right)\overline{v}\left(1\right)+\gamma^{\alpha }_k u^{'}\left(1\right)\hat{v}^{'}\left(1\right)\]

Thus,
\begin{equation} 
Y_{k,j}=c_{k,j}\Psi_{k,j}{\varphi }_{k}
\end{equation} 
where by  $\Psi_{k,j}\in \Lambda$  we denote the vector in parenthesis in the right side of (4.5). 

Set
\begin{equation} 
\Phi_{k,j}=\Psi _{k,j}\varphi_{k},
\end{equation} 
 obviously  $\Phi_{k,j} \in H_1$. $Y_{k,j}$ will form an orthonormal system of eigenvectors by putting in place of $c_{k,j}$ the norming constants  :
\begin{equation} 
{c_{k,j}}^2=\frac{1}{{{\left\|\Phi _{k,j}\right\|}^2}_1}=\frac{1}{{{\left\|{\mathrm{\Psi }}_{k,j}\right\|}^2}_{\Lambda}},
\end{equation} 
where:
\[\left\|\Phi_{k,j}\right\|^{2}_{1}=\left[\int^{1}_{0}\sin^{2} z_{k,j}tdt+\ H(z_{k,j})^2\int^{1}_{0}sh^2 z_{k,j}tdt+\right.
\]
\[
\left.+2H(z_{k,j})\int^{1}_{0}sh z_{k,j}t sin z_{k,j}tdt+H(z_{k,j})^2{\gamma _k}^{2\alpha}sh^2z_{k,j}\right]+
\]
\[ 
+2H(z_{k,j})\gamma _{k}^{2\alpha }shz_{k,j}sinz_{k,j}+{{\gamma }_k}^{2\alpha }\sin^{2}z_{k,j}+H(z_{k,j})^2\gamma_k^{2\alpha }{z_{k,j}}^2ch^2z_{k,j}+
\]
\begin{equation}
+2{z_{k,j}}^2{{\gamma }_k}^{2\alpha }H(z_{k,j})chz_{k,j}cosz_{k,j}+{{\gamma }_k}^{2\alpha }{z_{k,j}}^2{cos}^2z_{k,j}
\end{equation}
Here ${\left({\varphi }_k,{\varphi }_k\right)}_H=1$ was used.

\section{Some important relations}

 Recalling  that ${\gamma }_1\le {\gamma }_2\le \dots $ are  eigenvalues and $\left\{{\varphi }_k\right\},\ k=\overline{1,\infty }$    are orthonormal  eigenvectors of the operator $A$, in virtue of expansion theorem any $y(t)$ from $L_2\left(H,(0,1)\right)$ is expanded as
\[y\left(t\right)=\sum^{\infty }_{k=1}{\left(y\left(t\right){\varphi }_k,{\varphi }_k\ \right){\varphi }_k}\]
Denoting $(y\left(t\right){\varphi }_k,{\varphi }_k)=y_k(t)$, from (3.1)-(3.2)  we have the following spectral problem for scalar functions $y_k\left(t\right):$
\begin{equation} 
{l_ky_k\left(t\right)\equiv y}^{IV}_k\left(t\right)+{\gamma }_ky_k\left(t\right)=\lambda y_k\left(t\right)
\end{equation}
\begin{equation} 
y_k\left(0\right)=y_k^{''}\left(0\right)=0
\end{equation}
\begin{equation} 
{-y}_k^{'''}\left(1\right)=\lambda {\gamma }^{\alpha }_ky_k\left(1\right)
\end{equation}
\begin{equation} 
y_k^{''}\left(1\right)=\lambda {\gamma }^{\alpha }_ky_k^{'}\left(1\right)
\end{equation}
For each fixed  $k \left(k=1,\infty \right)$ denote the eigenvalues of that problem  by ${\lambda }_{k,j}$, and solutions of (5.1)  by $y_k\left(t,\lambda -{\gamma }_k\right).$ Obviously, ${\lambda }_{k,j},\ k=1,\infty ,\ j=1,\infty $ are the eigenvalues of problem (3.1)-(3.3) 

It is easy to see that vectors  $\left(y_{k,j}\left(t,{\lambda }_{k,j}-{\gamma }_k\right),{\gamma }^{\alpha}_ky_{k,j}\left(1,{\lambda }_{k,j}-{\gamma }_k\right),\ {\gamma }^{\alpha }_ky^{'}_{k,j}(1,{\lambda }_{k,j}-{\gamma }_k)\right)$,  form a set  of orthogonal  eigenvectors  of the  operator $L_{0k}$ associated with the scalar problem (5.1)-(5.4)  in space $\Lambda $ and acting as $L_{0k}\left(y_k\left(t,\lambda -{\gamma }_k\right),{\gamma }^{\alpha }_ky_k\left(1,\lambda -{\gamma }_k\right),\ {\gamma }^{\alpha }_ky^{'}_k(1,\lambda -{\gamma }_k)\right)=$ $=(l_ky_k(t,\lambda -{\gamma }_k), {-y}_k^{'''}\left(1,\lambda-\gamma_k \right),\ y_k^{''}\left(1,\lambda -\gamma_k\right))$

 Obviously, it coincide, with  $\Psi_{k,j}:$
\[\left(y_{k,j}\left(t,{\lambda }_{k,j}-{\gamma }_k\right),{\gamma }^{\alpha }_ky_{k,j}\left(1,{\lambda }_{k,j}-{\gamma }_k\right),\ {\gamma }^{\alpha }_ky^{'}_{k,j}(1,{\lambda }_{k,j}-{\gamma }_k)\right)=\Psi _{k,j}.
\]
  $\Upsilon_{k,j}=c_{k,j}\Psi_{k,j}$    and $Y_{k,j}=c_{k,j}\Phi_{k,j}$  are orthonormal eigenvectors of  problems and (5.1)-(5.4) and (3.1)-(3.3),  respectively.

Introduce the notations
\[ 
{\omega }_1(z)\equiv y_k^{'''}\left(1,\lambda -{\gamma }_k\right)+\lambda \gamma^{\alpha }_ky_k\left(1,\lambda -{\gamma }_k\right)=
\]
\begin{equation}
=y_k^{'''}\left(1,z^4\right)+{\left(z^4+{\gamma }_k\right)\gamma }^{\alpha }_ky_k\left(1,z^4\right)
\end{equation}
\[
{\omega }_2\left(z\right)\equiv y^{''}_k\left(1,\lambda -{\gamma }_k\right)-\lambda{\gamma }^{\alpha }_ky_k'\left(1,\lambda -{\gamma }_k\right)=
\]
\begin{equation} 
=y^{''}_k\left(1,z^4\right)-\left(z^4+{\gamma }_k\right){\gamma }^{\alpha }_ky_k'\left(1,z^4\right)
\end{equation}
The eigenvalues of (5.1)-(5.4) are defined from the system 
\begin{equation}
{\omega }_1\left(z\right)=0 
\end{equation}
\begin{equation}
{\omega }_2\left(z\right)=0 
\end{equation}
|Introduce the  following function:
\begin{equation} 
f_k\left(z\right)\equiv {y_k}^2\left(1,z^4\right)\left[\frac{{\omega }_1(z)}{y_k\left(1,z^4\right)}\right]-\ {y^{'}_k}^2(1,z^4)\left[\frac{{\omega }_2(z)}{y_k^{'}\left(1,z^4\right)}\right]
\end{equation}
Prove the next theorem which has an important role in deriving the trace formula.

 \begin{theorem}
  \[f^{'}_k\left(z_{k,j}\right)=4z^3_{k,j}{\left\|\Upsilon _{k,j}\right\|}^2=4z^3_{k,j}\frac{1}{c^2_{k,j}}\]
where $c^2_{k,j}$ are norming constants.
\end{theorem}
 \begin{proof}
 Let $y_k\left(t,\lambda -{\gamma }_k\right)$ and $y_k\left(t,{\lambda }_{k,j}-{\gamma }_k\right)$ be solutions of equation (5.1) with $\lambda $ and ${\lambda }_{k,j}$, respectively:
\begin{equation} 
y^{IV}_k\left(t,\lambda -{\gamma }_k\right)+{\gamma }_ky_k\left(t,\lambda -{\gamma }_k\right)=\lambda y_k\left(t,\lambda -{\gamma }_k\right)
\end{equation}
\begin{equation} 
y^{IV}_k\left(t,{\lambda }_{k,j}-{\gamma }_k\right)+{\gamma }_ky_k\left(t,{\lambda }_{k,j}-{\gamma }_k\right)={\lambda }_{k,j}y_k\left(t,{\lambda }_{k,j}-{\gamma }_k\right)
\end{equation}
Multiply (5.10) by $y_k\left(t,{\lambda }_{k,j}-{\gamma }_k\right)$, (5.11) $y_k\left(t,\lambda -{\gamma }_k\right)$, then subtract the second one from the first, integrate both sides of obtained relation  along $(0,1)$, and  to the obtained results add the term
\[
\left(\lambda -{\gamma }_k-({\lambda }_{k,j}-{\gamma }_k\right)){{\gamma }_k}^{\alpha }y_k\left(1,\lambda -{\gamma }_k\right)y_k\left(1,{\lambda }_{k,j}-{\gamma }_k\right)+
\]
\begin{equation} 
+\left(\lambda -{\gamma }_k-({\lambda }_{k,j}-{\gamma }_k\right))y_k'\left(1,\lambda -{\gamma }_k\right)y_k'\left(1,{\lambda }_{k,j}-{\gamma }_k\right){{\gamma }_k}^{\alpha }.
\end{equation}
Recall here the notations $\sqrt[4]{\lambda -{\gamma }_k}=z$ , $\sqrt[4]{{\lambda }_{k,j}-{\gamma }_k}=z_{k,j}$ from section 3. Note that addition of the last term is needed for  finding  the norm of the eigenvector of the operator $L_{0k}$ in the direct sum space $\Lambda _{1}=L_2\left(\left(0,1\right)\right)\oplus C^2$ as it will become clear in next derivations.

 Thus,
\[ 
N\equiv \int^1_0{y_k}^{IV}\left(t,z^4\right)y_k\left(t,z^4_{k,j}\right)dt-\int^1_0{y_k}^{IV}\left(t,z^4_{k,j}\right)y_k\left(t,z^4\right)dt+
\]
\[
+\left(z^4-z^4_{k,j}\right){\gamma }^{\alpha }_ky_k\left(1,z^4\right)y_k\left(1,z^4_{k,j}\right)+\left(z^4-z^4_{k,j}\right)y_k'\left(1,z^4\right)y_k'\left(1,z^4_{k,j}\right){\gamma }^{\alpha }_k=
\]
\[
=\left(z^4-z^4_{k,j}\right)\int^1_0{y_k\left(t,z^4\right)y_k\left(t,z^4_{k,j}\right)dt+}\left(z^4-z^4_{k,j}\right){\gamma }^{\alpha }_ky_k\left(1,z^4\right)y_k\left(1,z^4_{k,j}\right)+
\]
\begin{equation}
+\left(z^4-z^4_{k,j}\right)y_k'\left(1,z^4\right)y_k'\left(1,z^4_{k,j}\right){\gamma }^{\alpha }_k
\end{equation}

 Integration by parts on the left side of that  relation  yields
\[{N=y}_k'''\left(1,z^4\right)y_k(1,z^4_{k,j})-\]
\[{-y}_k'''\left(1,z^4_{k,j}\right)y_k\left(1,z^4\right)-y_k''\left(1,z^4\right)y_k'(1,z^4_{k,j})\ +y_k''\left(1,z^4_{k,j}\right)y_k'\left(1,z^4\right)+
\]\[
+\left(z^4-z^4_{k,j}\right){{\gamma }_k}^{\alpha }y_k\left(1,z^4\right)y_k\left(1,z^4_{k,j}\right)+\left(z^4-z^4_{k,j}\right)y_k'\left(1,z^4\right)y_k'\left(1,z^4_{k,j}\right){{\gamma }_k}^{\alpha }=\]
\[
=y_k\left(1,z^4_{k,j}\right)y_k\left(1,z^4\right)\left[\frac{y^{'''}_k\left(1,z^4\right)}{y_k\left(1,z^4\right)}-\frac{y^{'''}_k\left(1,z^4_{k,j}\right)}{y_k\left(1,z^4_{k,j}\right)}+\left(z^4-z^4_{k,j}\right){{\gamma }_k}^{\alpha }\right]-
\]
\begin{equation} 
-y_k'\left(1,z^4\right)y_k'\left(1,z^4_{k,j}\right)\left[\frac{y_k''\left(1,z^4\right)}{y_k'\left(1,z^4\right)}-\frac{y_k''\left(1,z^4_{k,j}\right)}{y_k'\left(1,z^4_{k,j}\right)}\ +\left(z^4-z^4_{k,j}\right){{\gamma }_k}^{\alpha }\right]\
\end{equation}

 For our further derivations  in that relation, we  consider $y_k\left(t,\lambda -{\gamma }_k\right)$  as a  realvalued function (it is so if we   take real values of  $\lambda $ ),thus  instead of $\int^1_0{y_k\left(t,\lambda -{\gamma }_k\right)\overline{y_k\left(t,{\lambda }_{k,j}-{\gamma }_k\right)}dt}$   one might take $\int^1_0{y_k\left(t,\lambda -{\gamma }_k\right)y_k\left(t,{\lambda }_{k,j}-{\gamma }_k\right)dt}=\int^1_0{y_k\left(t,z^4\right)y_k(t,z^4_{k,j})}dt.$

 Substituting the expression for $N$ from (5.14) into (5.13) with (5.12) in mind ,  dividing both sides of (5.13) by $z-z_{k,j}$ and passing to the limit as $z\to z_{k,j}$, we get
 
\[
4z_{k,j}^{3}\left[ \int_{0}^{1}{{y_{k}(t,z_{k,j}^{4})}^{2}dt}+{\gamma }%
_{k}^{\alpha }y_{k}^{2}\left( 1,z_{k,j}^{4}\right) +{\gamma }_{k}^{\alpha
}y_{k}^{2}\left( 1,z_{k,j}^{4}\right) +{\gamma }_{k}^{\alpha }{y}%
_{k}^{2}\left( 1,z_{k,j}^{4}\right) \right] =
\]%
\[
\underset{z\rightarrow z_{k,j}}{\lim }{\left( \frac{y_{k}\left(
1,z_{k,j}^{4}\right) y_{k}\left( 1,z^{4}\right) \left[ \frac{y_{k}^{^{\prime
\prime \prime }}\left( 1,z^{4}\right) }{y_{k}\left( 1,z^{4}\right) }-\frac{%
y_{k}^{^{\prime \prime \prime }}\left( 1,z_{k,j}^{4}\right) }{y_{k}\left(
1,z_{k,j}^{4}\right) }+z^{4}{\gamma }_{k}^{\alpha }\right] }{z-z_{k,j}}%
-\right. }
\]%
\[
\left. -\frac{y_{k}^{^{\prime }}\left( 1,z_{k,j}^{4}\right) y_{k}^{\prime
}\left( 1,z^{4}\right) \left[ \frac{y_{k}^{^{\prime \prime }}\left(
1,z^{4}\right) }{y_{k}^{\prime }\left( 1,z^{4}\right) }-\frac{%
y_{k}^{^{\prime \prime }}\left( 1,z_{k,j}^{4}\right) }{y_{k}^{^{\prime
}}\left( 1,z_{k,j}^{4}\right) }+z^{4}{\gamma }_{k}^{\alpha }\right] }{%
z-z_{k,j}}\right) =
\]%
\[
=y_{k}^{2}\left( 1,z_{k,j}^{4}\right) {\left[ \frac{y_{k}^{^{\prime \prime
\prime }}\left( 1,z^{4}\right) }{y_{k}\left( 1,z^{4}\right) }\right] }%
^{^{\prime }}\left. {}\right\vert _{z=z_{k,j}}+4z_{k,j}^{3}{{\gamma }_{k}}%
^{\alpha }{y_{k}}^{2}\left( 1,z_{k,j}^{4}\right) -
\]%
\[
-y_{k}^{\prime 2}\left( 1,z_{k,j}^{4}\right) \left[ \frac{%
y_{k}^{^{\prime \prime }}\left( 1,z^{4}\right) }{y_{k}\left( 1,z^{4}\right)} %
\right] ^{^{\prime }}\left. {}\right\vert _{z=z_{k,j}}-4z_{k,j}^{3}\gamma
_{k}^{\alpha }y_{k}^{^{\prime }}{}^{2}\left( 1,z_{k,j}^{4}\right) =
\]%
\begin{equation}
=y_{k}^{2}\left( 1,z_{k,j}^{4}\right) {\left[ \frac{y_{k}^{^{\prime \prime
\prime }}\left( 1,z^{4}\right) }{y_{k}\left( 1,z^{4}\right) }+z^{4}{\gamma }%
_{k}^{\alpha }\right] }^{^{\prime }}\left. {}\right\vert _{z=z_{k,j}}-{%
y_{k}^{^{\prime }}}^{2}\left( 1,z_{k,j}^{4}\right) {\left[ \frac{%
y_{k}^{^{\prime \prime }}\left( 1,z^{4}\right) }{y_{k}^{\prime }\left(
1,z^{4}\right) }+z^{4}{\gamma }_{k}^{\alpha }\right] }^{^{\prime }}\left.
{}\right\vert _{z=z_{k,j}}
\end{equation}
(derivatives  of expressions within square bracket  in the last relation  are taken with respect to $z$)

Note that $\int\limits^{1}_{0}y_k(t,z^{4}_{k,j})^{2}dt+\gamma_{k}^{\alpha}y_{k}^{2}(1,z^{4}_{k,j})+\gamma_{k}^{\alpha }y^{'2}_{k}(1,z^{4}_{k,j})$
 standing on the left side of (5.15) is square of the norm of  eigenvectors of an operator associated with problem (5.1)-(5.4) in $\Lambda$:
\begin{equation}
\int_{0}^{1}y_{k}(t,z_{k,j}^{4})^{2}dt+\gamma _{k}^{\alpha }{y_{k}}%
^{2}(1,z_{k,j}^{4})+\gamma _{k}^{\alpha }{y}_{k}^{2}(1,z_{k,j}^{4})^{^{
\prime }}=\left\Vert \Psi _{k,j}\right\Vert _{\Lambda }^{2}
\end{equation}
Using (5.16)  on the left side of (5.15)  and notations (5.5), (5.6), we arrive at
\begin{equation}
4z_{k,j}^{3}\left\Vert \Psi _{k,j}\right\Vert _{\Lambda
}^{2}=y_{k}^{2}\left( 1,{z^{4}}_{k,j}\right) {\left[ {\omega }_{1}(z)%
\right] ^{^{\prime }}}\left. {}\right\vert _{{_{z=z_{k,j}}}}{-y}_{k}^{2}{(1,{%
z^{4}}_{k,j})\left[ {\omega }_{2}\left( z\right) \right] }^{^{\prime
}}\left. \right\vert _{{_{z=z_{k,j}}}}
\end{equation}
 Obviously, the solution of problem (5.1) satisfying (5.2) is
\[y_k\left(t,\lambda -{\gamma }_k\right)=c_{1k}sh\sqrt[4]{\lambda -{\gamma }_k}t+c_{2k}sin\sqrt[4]{\lambda -{\gamma }_k}t \]
or
\begin{equation} 
y_k\left(t,z^4\right)=c_{1k}shzt+c_{2k}sinzt
\end{equation}
For this function to be the first component of eigenvector ${\mathrm{\Psi }}_{k,j}$ of problem (5.1)-(5.4), the function  $y_k\left(t,z^4\right)$ must satisfy also (5.3) and (5.4) or equivalently (5.7), (5.8).  Substituting it into  (5.15), we get again (3.11), (3.12) .

 Obviously, writing $y_k\left(t,z^4\right)$from (5.8) with $c_{1k}$ defined from  (4.2)  into (5.4)  yields (3.11) from which ${\lambda }_{k,j}$ are found as  ${\lambda }_{k,j}=z^4_{k,j}+{\gamma }_k$

 Evaluate the derivative of $f_k\left(z\right)$ at $z_{k,j}$
\[
f^{'}_k\left(z_{k,j}\right)=\left({y_k}^2\left(1,z^4\right)\right)^{'}\left. \right\vert _{{_{z=z_{k,j}}}}\left[\frac{ \omega_{1}\left(z_{k,j}\right)}{y_k\left(1,z^{4}_{k,j}\right)}\right]+{y_k}^2\left(1,z^4_{k,j}\right){\left[\frac{{\omega }_1(z)}{y_k\left(1,z^4\right)}\right]}^{'}\left. \right\vert _{z=z_{k,j}} -
\]
\[
-\left(y^{'2}_k(1,z^4)\right)^{'}\left. \right\vert _{{_{z=z_{k,j}}}}\left[\frac{\omega _{2}\left(z_{k,j}\right)}{y_k\left(1,z^4_{k,j}\right)}\right]-y^{'2}_{k}(1,z^4_{k,j})\left[\frac{\omega _{2}(z)}{y_k^{'}\left(1,z^4\right)}\right]^{'}\left. \right\vert _{{_{z=z_{k,j}}}}
\]
which in virtue of ${\omega }_1\left(z_{k,j}\right)=0,\ {\omega }_2\left(z_{k,j}\right)=0$ yields
\begin{equation} 
f^{'}_k\left(z_{k,j}\right)=y_k^2\left(1,z^4_{k,j}\right){\left[{\omega }_1(z)\right]}^{'}\left. {}\right\vert _{{_{z=z_{k,j}}}}- y^{'2}_k(1,z^4_{k,j})\left[{\omega }_2\left(z\right)\right]^{'}\left. {}\right\vert_{{_{z=z_{k,j}}}}
\end{equation}
From which with  (5.16) in mind we get
\begin{equation} 
f^{'}_k\left(z_{k,j}\right)=4 z^{3}_{k,j}\left\|\Psi_{k,j}\right\|^{2}_{\Lambda }=4z^{3}_{k,j}\frac{1}{c^2_{k,j}}
\end{equation}
where
\[c^2_{k,j}=\frac{1}{\left\|\Psi _{k,j}\right\|^{2}_{\Lambda}}=\frac{1}{\left\|\Phi _{k,j}\right\|^{2}_1}\]
which completes the proof.

 Thus, (5.18), (5.20) relates the characteristic determinant or  ${\omega }_1\left(z\right)$ and ${\omega }_2\left(z\right)$  whose comman zeros define the eigenvalues and norms of orthogonal eigenvectors. The function $f_k\left(z\right)$ and its analogous for the problem  (5.1), (5.2) but with  boundary conditions different than (5.3), (5.4)  have essential role in our derivation of trace formula  described in the next section.

 If  $c_{1k}$ is defined from   (5.8),  then $f_k\left(z\right)$ from   (5.18) is simplified to the form
\begin{equation} 
f_k\left(z\right)\equiv y_k\left(1,z^4\right){\omega }_1(z)
\end{equation}
$\Upsilon_{k,j}=c_{k,j}\Psi_{k,j}$  are  orthonormal eigenvectors of the operator $L_{0k}$ associated with problem (5.1)-(5.4) in the space $\Lambda,$ and
\[{Y_{k,j}=c}_{k,j}{\Phi }_{k,j}=c_{k,j}{\Psi}_{k,j}{\varphi }_k\]
are orthonormal eigenvectors of the operator $L_0$  associated  with problem (3.1)-(3.3) in $H_1$. 
\end{proof}

\section{Evaluation of regularized trace}

Before passing to derivations, put on $q(t)$ the next condition:
\begin{equation} 
\sum^{\infty }_{k=1}{{\left(q\left(t\right){\varphi }_k,{\varphi }_k\right)}_H<\infty }
\end{equation} 

From Theorem 2.3 and Note 2.1 the operator $L_1=L^1_0+Q$ is discrete. Recall that eigenvalues of  $L_0$ and $L_1$ are denoted  by ${\lambda }_1\le {\lambda }_2\le \dots $ and ${\mu }_1\le {\mu }_2\dots,$ respectively. In virtue of  Theorem 3.1 from section 3 and Theorem 1 from [10](application to our case is justified as in our work [9])
\begin{equation} \label{GrindEQ__5_1_}
{\mathop{\mathrm{lim}}_{m\to \infty } \sum^{n_m}_{n=1}{\left({\mu }_n-{\lambda }_n-{\left(QY_{k_nj_n},Y_{k_nj_n}\right)}_1\right)=}\ }0
\end{equation}
for some subsequence of natural numbers $\left\{n_m\right\}$ satisfying conditions of Lemma3 from [10].

In virtue of the asymptotic formula for $z_{k,j}$ the next lemma  is proved (proof is similar to one , for example in [9], Lemma 4.2).

\begin{lemma} The series
\[\sum^{\infty }_{j=1}{\sum^{\infty }_{k=1}{{\left(QY_{k,j},Y_{k,j}\right)}_1}}\]
is absolutely convergent.
\end{lemma}
 From \eqref{GrindEQ__5_1_} and Lemma5.1
\begin{equation} \label{GrindEQ__5_2_}
{\mathop{\mathrm{lim}}_{m\to \infty } \sum^{n_m}_{n=1}{\left({\mu }_n-{\lambda }_n\right)=}\ }{\mathop{\mathrm{lim}}_{m\to \infty } \sum^{n_m}_{n=1}{{\left(QY_{k_nj_n},Y_{k_nj_n}\right)}_1}=\sum^{\infty }_{j=1}{\sum^{\infty }_{k=1}{{\left(QY_{k,j},Y_{k,j}\right)}_1}}}
\end{equation}
Denote by $\sum\limits^{\infty ' }_{n=1}{\left({\mu }_n-{\lambda }_n\right)}$ the limit on the left side of (6.3) and call it regularized   trace of $L_1$.
\begin{equation} 
\sum^{\infty ' }_{n=1}{\left({\mu }_n-{\lambda }_n\right)}=\sum^{\infty }_{j=1}{\sum^{\infty }_{k=1}{{\left(QY_{k,j},Y_{k,j}\right)}_1}}
\end{equation}

From  lemma 6.1 and from (6.4), with (4.5)  in mind, we get:
\[\sum^{\infty ' }_{n=1}{\left({\mu }_n-{\lambda }_n\right)}=\sum^{\infty }_{j=1}{\sum^{\infty }_{k=1}{{\left(QY_{k,j},Y_{k,j}\right)}_1}}=\sum^{\infty }_{j=1}{\sum^{\infty }_{k=1}{c^2_{k,j}\int^1_0{{q_k\left(t\right)\ y}^2_k(t,z^4_{k,j})dt}}}=
\]\[
=\sum^{\infty }_{j=1}{\sum^{\infty }_{k=1}{c^2_{k,j}\int^1_0{q_k\left(t\right)\left[{sin}^2z_{k,j}t+2H\left(z_{j,k}\right)shz_{k,j}tsinz_{k,j}t+{H\left(z_{j,k}\right)}^2{sh}^2z_{k,j}t\right]dt}}},
\]
where 
\begin{equation}
q_k\left(t\right)=\left(q\left(t\right){\varphi }_k,{\varphi }_k\right)  \end{equation}

Without loss of generality, putting $\int^1_0{q_k\left(t\right)dt=0}$ with Theorem 5.1 in mind we come to
 \[
\sum^{\infty '}_{n=1}\left({\mu }_n-{\lambda }_n\right)=
\]
\small{
\begin{equation}
=\sum^{\infty }_{j=1}{\sum^{\infty }_{k=1}{c^2_{k,j}\frac{\int^1_0{q_k\left(t\right)\ \left[-cos2z_{k,j}t +4H\left(z_{j,k}\right)shz_{k,j}tsinz_{k,j}t+{H(z_{j,k})}^2ch2z_{k,j}t\right]}dt}{2}}}
\end{equation}}
Consider, the $N$-th partial sum of the inner series: 
\begin{equation} 
\sum\limits^N_{k=1}c^2_{k,j}\frac{\int\limits^{1}_{0}q_{k}\left(t\right)\left[-cos2z_{k,j}t+4H\left(z_{j,k}\right)shz_{k,j}t\sin z_{k,j}t+{H(z_{j,k})}^2ch 2 z_{k,j} \right]dt}{2} 
\end{equation}
or in a more compact form $\sum\limits^N_{k=1}{c^2_{k,j}\int\limits^1_0{q_k\left(t\right)y^2_k(t,{z_{k,j}}^4)dt}}$

 Our aim in that section is to find the sum of the series in (6.6). For that sake in our previous works , for example, [8,9],  for evaluating the value of the right side of (6.4)  we use Cauchy's residue theorem, further tending contour of integration to infinity  and using asymptotic formulas for the integrand. Namely, each time we have selected  a function of a complex variable with poles at $z_{j,k}$ (zeros of characteristic determinant): they are the functions, whose denominators are defined by the  characteristic determinants $\Delta(z)$ corresponding to the problem (whose equivalent in the present work is equation  (3.14)) and numerators are suggested by numerators of series the sum of  which to be evaluated (here numerator of (6.7)). Usually, residues at poles of that function give terms of  sum analog of which here is (6.6) which indicates on some relation between characteristic determinant of the associated operator and norming constants.  Further, using asymptotic formulas found for $z_{k,j}$ on the integration contour, we arrived at the desired formulas.

 In thus work, since the norming constants defined by (4.8) and (4.9) and  $\Delta (z)$  from (3.14) have too long expressions, and that is why to manipulate with them by using the above indicated methods is impossible.

 For that reason , by (5.19) and  (5.20) we  establish the indicated above relation between ${\omega }_1\left(z\right),\ {\omega }_2(z)$ and norming constants existence of which was intuitively clear for us in  all previous works. Remind that by determining, for example, $c_{1k}$ from ${\omega }_1\left(z\right)=0$ and substituting in ${\omega }_2\left(z\right)=0$  yields an equation equivalent to $\mathrm{\Delta }\left(z\right)=0$ (see (3.14)). Note that this a is more general method and might be used in  studies of regularized traces in the  future.

Interchange the integral and the sum in (6.7) and denote by $S_N(t)$ the following  expression
\begin{equation} 
S_N(t)=\sum^N_{j=1}{c^2_{k,j}\ y^2_k(t,z^4)}
\end{equation}
Now using (5.20),(5.21), we see that the following functions of a complex variable z
\begin{equation} 
F_k\left(z,t\right)=\frac{4z^3y^2_k(t,z^4)}{f_k\left(z\right)}=\frac{4z^3y^2_k(t,z^4)}{y_k\left(1,z^4\right){\omega }_1\left(z\right)\mathrm{\ }}\
\end{equation}
 have poles at common roots $z=z_{k,j}$ of system (5.7),(5.8)  and the residues of  $F_k\left(z,t\right)$ at these poles  are the terms of the sum  (6.7):
\begin{equation}
{res}_{z=z_{k,j}}F_k\left(z,t\right)=c^2_{k,j}\ y^2_k(t,z^4) \end{equation}
In virtue of (6.6), (6.10) we arrive at the next lemma

\begin{lemma}
\begin{equation} 
\sum^{'\infty }_{n=1}{\left({\mu }_n-{\lambda }_n\right)=\sum^{\infty }_{k=1}{\sum^{\infty }_{j=1}{\int^1_0{{res}_{z=z_{k,j}}F_k\left(z,t\right)q_k(t)dt}}}}
\end{equation}
where $F_k\left(z,t\right)$ is defined by (5.8)
\end{lemma}
(4.8) and (4.9) show that expressions for  norming constants $c_{k,j}$ are  too long. In our previous works, we  write $c_{k,j}$ in an open form and write concrete form for expressions analogous to $y_k\left(1,\lambda \right){\omega }_1(z)$ in  choice of  $F_k\left(z,t\right)$ . But since here  those functions have quite long expressions, it is impossible to arrive at results or manipulate  formulas  writing them in an open form. That is why there arise a need in derivations (5.10)-(5.17) which let us  relate the norms of eigenvectors with ${\omega }_1\left(z\right),\ {\omega }_2\left(z\right)$ or characteristic determinant  $\Delta (z)$ whose zeros are related  to eigenvalues  of  $L_0$ .

 But the function $F_k\left(z,t\right)$ together with $z_{k,j}$ has poles also at zeros of $y_k\left(1,z^4\right)$. Denoting the zeros of $y_k\left(1,z^4\right)$ by $z={\beta }_{k,j}$ we have
\[{res}_{z={\beta }_{k,j}}F_k\left(z,t\right)=\frac{4{{\beta }_{k,j}}^3y^2_k(t,{{\beta }_{k,j}}^4)}{\dot{y_k}\left(1,{{\beta }_{k,j}}^4\right){\omega }_1\left({\beta }_{k,j}\right)}\]
 where dot indicate derivative with respect  to $z$.
 
But from $y\left(1,\ {{\beta }_{k,j}}^4\right)=0$
\begin{equation} 
{res}_{z={\beta }_{k,j}}F_k\left(z,t\right)=\frac{4{{\beta }_{k,j}}^3y^2_k(t,{{\beta }_{k,j}}^4)}{\dot{y_k}\left(1,{{\beta }_{k,j}}^4\right)y'''\left(1,\ {{\beta }_{k,j}}^4\right)}.
\end{equation}

Note that ${{\beta }_{k,j}}^4+{\gamma }_k$ are the eigenvalues of problem (5.1),(5.2)  and (6.13),(6.14) for each fixed $k$
\begin{equation} 
y_k\left(1\right)=0
\end{equation}
\begin{equation} 
y_k"\left(1\right)-\lambda {\gamma }^{\alpha }_ky_k'\left(1\right)=0
\end{equation}
and the collection ${\left\{{{\beta }_{k,j}}^4+{\gamma }_k\right\}}^{\infty }_{k,j=1}$ are the eigenvalues of problem (3.1), (3.2), and (6.15), (6.16)
\begin{equation} 
y\left(1\right)=0
\end{equation}
\begin{equation} 
y"(1)-\lambda A^{\alpha }y'(1)=0
\end{equation}
Selecting rectangular  contour $l_N$ including inside it  $z_{k,j}$ and ${\beta }_{k,j}$  for each fixed $k$  and $ j=\overline{1,N}$ (we can choose  such a contour because of asymptotics of $z_{k,j}$, ${\beta }_{k,j}$ ), see, for example, [8,9]) and applying the Cauchy theorem  about residues we have
\begin{equation} 
\sum^N_{j=1}{{res}_{z=z_{k,j}}F_k\left(z,t\right)}=-\sum^N_{j=1}{{res}_{z={\beta }_{k,j}}F_k\left(z,t\right)}+\int_{l_N}{F_k\left(z,t\right)dz}
\end{equation}
Multiplying by  $q_k(t)$, integrating along [0,1] and passing to limit in (6.17)  as $N\to \infty $ yields
\[ 
\sum^{\infty }_{j=1}{\int^1_0{{res}_{z=z_{k,j}}F_k\left(z,t\right)q_k(t)dt}}=
\]
\begin{equation}
{-\sum^{\infty }_{j=1}{\int^1_0{{res}_{z={\beta }_{k,j}}F_k\left(z,t\right)q_k(t)dt}}+\mathop{\mathrm{lim}}_{N\to \infty } \int_{l_N}{\int^1_0{F_k\left(z,t\right)dz}}\ }q_k(t)dtdz.
\end{equation}
 By using the asymptotics of $F_k\left(z,t\right)$ for large $\left|z\right|$  values it can be shown that  as  $N$ tends to infinity, the integral along the extended contours  approaches zero. So, with (6.11) and (6.12) in mind
\[
\sum^{\infty'}_{n=1}{\left({\mu }_n-{\lambda }_n\right)=\sum^{\infty }_{k=1}{\sum^{\infty }_{j=1}{\int^1_0{{res}_{z=z_{k,j}}F_k\left(z,t\right)q_k(t)dt}}}}=\]
\begin{equation} 
=-\sum^{\infty }_{k=1}{\sum^{\infty }_{j=1}{\int^1_0{{res}_{z={\beta }_{k,j}}F_k\left(z,t\right)q_k(t)dt}}}=-\sum^{\infty }_{j=1}{\frac{4{\beta }^3_{k,j}\int^1_0{y^2_k(t,{\beta }^4_{k,j})q_k(t)dt}}{2\dot{y_k}\left(1,{{\beta }_{k,j}}^4\right)y^{'''}_k\left(1,{{\beta }_{k,j}}^4\right)}}
\end{equation}
Let $L_{11}=$ $L_{01}+Q,$ where  $L_{01}$ is an operator corresponding to (3.1) (3.2),(6.15),(6.16)  which is  defined in space $H_2=L_2\left(H,\left(0,1\right)\right)\oplus H$ of vectors $Y= (y(t),y_{1}), Z = (z(t), z_{1})$ where $y_1,z_1 \in H$  , with scalar product defined as, ${\left(Y,Z\right)}_2={\left(y\left(t\right),z(t)\right)}_{L2(H,(01))}$+$\left({A^{-\frac{\alpha }{2}}y}_1,{A^{-\frac{\alpha }{2}}z}_1\right)$,$D\left(L_{01}\right)=\left\{Y\in D\left(L^*_0\right),y\left(1\right)=0,y_1=A^{\alpha }y'(1)\right\}$, $L_{01}Y=\left\{ly\left(t\right),\ y"(1)\right\}$ and $Q$ this time is defined as $QY=\left\{q\left(t\right)y\left(t\right),\ 0\right\}$. Moreover, denote by $L_{01k}$ the operator defined by   $L_{01k}\left(y_k\left(t\right),{\gamma }^{\alpha }_ky_k'\left(1\right)\right)=\left\{l_ky_k\left(t\right),y^{''}_k\left(1\right)\ \right\}$ in space $\Lambda _2=L_2\left(0,1\right)\oplus C$

\textbf{Theorem 6.1}
\[{res}_{z={\beta }_{k,j}}F_k\left(z,t\right)=-c^2_{k,j}\ y^2_k(t,{\beta }^4_{k,j})\]
\textit{with   $c^2_{k,j}=\frac{1}{{{\left\|\Phi _{k,j}\right\|}^2}_2}=\frac{1}{{{\left\|\Psi _{k,j}\right\|}^2}_{2}}$, where $\left\{\Phi_{k,j}\right\}$ this time are orthogonal eigenvectors of the operator $L_{01}$ in $H_2$ associated with problem (3.1) (3.2),(6.15),(6.16), and $\Psi _{k,j}$ are orthogonal eigenvectors of problem (5.1) (5.2),(6.13),(6.14).}

\textit{To prove it, from the right side of (6.19) we see that, it must be shown that}
\[c^2_{k,j}=-\frac{4{\beta }^3_{k,j}}{\dot{y_k}\left(1,{{\beta }_{k,j}}^4\right)y^{'''}_k\left(1,{{\beta }_{k,j}}^4\right)}\]

\begin{proof} Again multiplying (5.10) by $y_k\left(t,{\lambda }_{k,j}-{\gamma }_k\right)$, (5.11) $y_k\left(t,\lambda -{\gamma }_k\right)$, then subtracting the second relation  from the first, integrating the both sides of obtained the relation  along (0,1), adding to the obtained results the term 
\begin{equation}
\left(\lambda -{\lambda }_{k,j}\right)y_k'\left(1,\lambda -{\gamma }_k\right)y_k'\left(1,{\lambda }_{k,j}-{\gamma }_k\right){{\gamma }_k}^{\alpha }
\end{equation}
keeping  in mind
\begin{equation} 
y_k\left(1,{\lambda }_{k,j}-{\gamma }_k\right)=y_k\left(1,{{\beta }_{k,j}}^4\right)=0
\end{equation}
(we again keep notations $\sqrt[4]{\lambda -{\gamma }_k}=z,\ \sqrt[4]{{\lambda }_{k,j}-{\gamma }_k}{=\beta }_{k,j}$ , where  this time ${\lambda }_{k,j}={{\beta }_{k,j}}^4+{\gamma }_k$\ ,\ $j=1,\infty$  are eigenvalues  of the problem  (5.1),(5.2),(6.13), (6.14) we get:

\[E\equiv \int^1_{0}y_{k}^{IV}\left(t,z^4\right)y_k\left(t,\beta _{k,j}^{4}\right)dt-
\]
\[
-\int^1_{0}y_{k}^{IV}\left(t,\beta _{k,j}^{4}\right)y_k\left(t,z^4\right)dt+\left(\lambda -\lambda _{k,j}\right)y^{'}_k\left(1,z^{4}\right)y_k^{'}\left(1,\beta _{k,j}^{4}\right)\gamma _k^{\alpha }=\]
\[=\left(z^4-{{\beta }_{k,j}}^4\right)\int^1_0{y_k\left(t,z^4\right)y_k\left(t,{{\beta }_{k,j}}^4\right)dt+}\left(z^4-{{\beta }_{k,j}}^4\right)y^{'}_k\left(1,z^4\right)y_k^{'}\left(1,{{\beta }_{k,j}}^4\right){\gamma }^{\alpha }_k\]
Integration by parts gives
\[{{E=y}_k^{'''}\left(1,z^4\right)y_k(1,{{\beta }_{k,j}}^4)-y}_k^{'''}\left(1,{{\beta }_{k,j}}^4\right)y_k\left(1,z^4\right)-y_k^{''}\left(1,z^4\right)y_k^{'}(1,{{\beta }_{k,j}}^4)+
\]\[
+y_k^{''}\left(1,{{\beta }_{k,j}}^4\right)y_k^{'}\left(1,z^4\right) +\left(z^4-{{\beta }_{k,j}}^4\right)y^{'}_k\left(1,z^4\right)y_k'\left(1,{{\beta }_{k,j}}^4\right){{\gamma }_k}^{\alpha }= 
\]\[
={-y}_k^{'''}\left(1,{{\beta }_{k,j}}^4\right)y_k\left(1,z^4\right)-y_k^{''}\left(1,z^4\right)y_k^{'}(1,{{\beta }_{k,j}}^4)+y_k^{''}\left(1,{{\beta }_{k,j}}^4\right)y_k^{'}\left(1,z^4\right)\ +
\]
\[
+\left(z^4-{{\beta }_{k,j}}^4\right)y^{'}_k\left(1,z^4\right)y_k^{'}\left(1,{{\beta }_{k,j}}^4\right){{\gamma }_k}^{\alpha }=\]
\[ 
=-y^{'''}_k\left(1,{{\beta }_{k,j}}^4\right)\left[y_k\left(1,z^4\right)-y_k\left(1,{{\beta }_{k,j}}^4\right)\right]-
\]
\[
-y^{'}_k\left(1,z^4\right)y_k'\left(1,{{\beta }_{k,j}}^4\right)\left[\frac{y_k^{''}\left(1,z^4\right)}{y_k^{'}\left(1,z^4\right)}-\frac{y_k^{''}\left(1,{{\beta }_{k,j}}^4\right)}{y_k^{'}\left(1,{{\beta }_{k,j}}^4\right)}\right]+
\]
\begin{equation}
+\left(z^4-{{\beta }_{k,j}}^4\right)y^{'}_k\left(1,z^4\right)y_k'\left(1,{{\beta }_{k,j}}^4\right){{\gamma }_k}^{\alpha }
\end{equation}
So,
\[
\left(z^4-{{\beta }_{k,j}}^4\right)\int^1_0{y_k\left(t,z^4\right)y_k\left(t,{{\beta }_{k,j}}^4\right)dt}+\left(z^4-{{\beta }_{k,j}}^4\right)y^{'}_k\left(1,z^4\right)y^{'}_k\left(1,{{\beta }_{k,j}}^4\right){{\gamma }_k}^{\alpha }=
\]
\[
={-y}^{'''}_k\left(1,{{\beta }_{k,j}}^4\right)\left[y_k\left(1,z^4\right)-y_k\left(1,{{\beta }_{k,j}}^4\right)\right]-y^{'}_k\left(1,z^4\right)y_k^{'}\left(1,{{\beta }_{k,j}}^4\right)\left[\frac{y_k^{''}\left(1,z^4\right)}{y_k'\left(1,z^4\right)}-\frac{y_k^{''}\left(1,{{\beta }_{k,j}}^4\right)}{y_k^{'}\left(1,{{\beta }_{k,j}}^4\right)}\right]+
\]
\begin{equation} 
+\left(z^4-{{\beta }_{k,j}}^4\right)y^{'}_k\left(1,z^4\right)y_k^{'}\left(1,{{\beta }_{k,j}}^4\right){{\gamma }_k}^{\alpha }
\end{equation}
Recall that  the term $y_k\left(1,{{\beta }_{k,j}}^4\right)$ can  appear on the left side of (6.23) because of (6.21)

 Note here ${\lambda }_{k,j}={{\beta }_{k,j}}^4+{\gamma }_k$.

 Dividing the  both sides of (6.23) by $z-{\beta }_{k,j}$ , passing to the limit as $z\to {\beta }_{k,j}$ and denoting orthogonal eigenvectors of the  problem (5.1),(5.2),(6.13),(6.14)  again by $\Psi_{k,j}$,  we get \[
4\beta _{k,j}^{3}\left\Vert \Psi _{k,j}\right\Vert _{\Lambda
_{2}}^{2}=-y_{k}^{^{\prime \prime \prime }}\left( 1,\beta _{k,j}^{4}\right)
\dot{y_{k}}\left( 1,\beta _{k,j}^{4}\right) -
\]%
\[
-y_{k}^{^{\prime }}\left( 1,\beta _{k,j}^{4}\right)
^{2}\lim\limits_{z\rightarrow \beta _{k,j}}\left[ \frac{\frac{%
y_{k}^{^{\prime \prime }}\left( 1,z^{4}\right) }{y_{k}^{^{\prime }}\left(
1,z^{4}\right) }-\frac{{y_{k}^{^{\prime \prime }}\left( 1,{{\beta }_{k,j}^{4}%
}\right) }}{{y_{k}^{^{\prime }}\left( 1,{{\beta }_{k,j}}^{4}\right) }}}{{%
y_{k}^{^{\prime }}\left( 1,{{\beta }_{k,j}}^{4}\right) z-{\beta }_{k,j}-}%
\frac{z^{4}-{{\beta }_{k,j}^{4}}}{z-{{\beta }_{k,j}{\gamma }_{k}^{\alpha }}}}%
\right] =
\]%
\begin{equation}
=-y_{k}^{^{\prime \prime \prime }}\left( 1,\beta _{k,j}^{4}\right) y_{k}{%
\left( 1,z^{4}\right) }^{^{\prime }}|_{z=\beta _{k,j}}-{y_{k}^{^{\prime
}}\left( 1,{{\beta }_{k,j}}^{4}\right) }^{2}\left[ \frac{y_{k}^{^{\prime
\prime }}\left( 1,z^{4}\right) }{y_{k}^{^{\prime }}\left( 1,z^{4}\right) }%
-z^{4}{{\gamma }_{k}}^{\alpha }\right] ^{^{\prime }}|_{z=\beta _{k,j}}
\end{equation}
If $c_{2k}$  is defined from ${\omega }_2\left(z\right)=0,$ then in   right side of (6.24)   $\frac{y^{''}_k\left(1,z^4\right)}{y_k'\left(1,z^4\right)}-z^4{{\gamma }_k}^{\alpha }\equiv 0$   and    (6.24)  simplifies to
\begin{equation} 
4{{\beta }_{k,j}}^3{{\left\|\Psi _{k,j}\right\|}^2}_{\Lambda _2}={-y}^{'''}_k\left(1,{{\beta }_{k,j}}^4\right)y_k{\left(1,z^4\right)}^{'}|_{z={\beta }_{k,j}}
\end{equation}
or
\begin{equation} 
\left\|\Phi _{k,j}\right\|^{2}_{2}=\left\|\Psi _{k,j}\right\|^{2}_{\Lambda _2}=\frac{-y^{'''}_{k}\left(1,\beta _{k,j}^{4}\right)y_k{(1,z^4)}^{'}|_{z=\beta_{k,j}}}{4\beta _{k,j}^{3}}
\end{equation}
We have (see (5.12), (6.21) )
\begin{equation} 
{\omega }_1\left({\beta }_{k,j}\right)={-y}^{'''}_k\left(1,{\lambda }_{k,j}\right)
\end{equation}
So, for norming constants of the problem (3.1),(3.2),(6.15),(6.16) 
\begin{eqnarray}
\frac{1}{c_{k,j}^{2}} &=&{\left\Vert \Phi _{k,j}\right\Vert }_{H_{2}}^{2}={%
\left\Vert \Psi _{k,j}\right\Vert }_{\Lambda _{2}}^{2}=\frac{\omega _{1}({{%
\beta }_{k,j})y}_{k}(1,z^{4})^{\prime }|_{z={\beta }_{k,j}}}{4{{\beta }%
_{k,j}^{3}}}= \\
&=&\frac{-y_{k}^{\prime \prime \prime }(1,4{{\beta }_{k,j}^{4})y}%
_{k}(1,z^{4})|_{z={\beta }_{k,j}}}{4{{\beta }_{k,j}^{3}}}
\end{eqnarray}
From (6.12) and  (6.28) 

\[ {res}_{z={\beta }_{k,j}}F_k\left(z,t\right)=\frac{4{{\beta }_{k,j}}^3y^2_k(t,{\beta }^4_{k,j})}{\dot{y_k}\left(1,{{\beta }_{k,j}}^4\right){\omega }_1\left({\beta }_{k,j}\right)}=\frac{{4{{\beta }_{k,j}}^3 y^2_k(t,{\beta }^4_{k,j})\ }}{\dot{{2y}_k}\left(1,{{\beta }_{k,j}}^4\right)y^{'''}_k\left(1{{,\beta }_{k,j}}^4\right)}=
\]
\begin{equation} 
-\frac{y^2_k(t,{\beta }^4_{k,j})}{{{\left\|\Psi _{k,j}\right\|}^2}_1}={-c}^2_{k,j}y^2_k(t,{\beta }^4_{k,j})
\end{equation}
$c_{k,j}$ (if k is fixed )are now norming constants of the problem (5.11), (5.12), (6.13), (6.14) or for varying k of the problem (3.1),(3.2), (6.15), (6.16).

 Denoting the eigenvalues of $L_{01}$  and $L_{11}$  by ${\lambda }_{n1\ },\ {\mu }_{n1}$ respectively, we have  for the regularized trace of $L_{11}$  as in  (6.4)
\begin{equation} 
\sum^{\infty'}_{n=1}{\left({\mu }_{n1}-{\lambda }_{n1}\right)}=\sum^{\infty }_{j=1}{\sum^{\infty }_{k=1}{{\left(QY_{k,j},Y_{k,j}\right)}_2}},
\end{equation}
where $\left\{Y_{k,j}\right\}$ are  now orthonormal eigenvectors of the operator $L_{01}$.

In virtue of (6.31) and application   Theorem 6.1 to $L_{01}$ yields
\begin{equation} 
\sum^{\infty '}_{n=1}{\left({\mu }_{n1}-{\lambda }_{n1}\right)}=\sum^{\infty }_{k=1}{\sum^{\infty }_{j=1}{c^2_{k,j}\int^1_0{q_k(t)y^2_k(t,{\beta }^4_{k,j})dt}}}=-\sum^{\infty }_{k=1}{\sum^{\infty }_{j=1}{\int^1_0{{}^{res}_{z={\beta }_{k,j}}{F_k}(z,t)q_k(t)dt}}}
\end{equation}
By comparing (6.32) and (6.19) we get
\end{proof}

\textbf{Corollary 6.1.}
\[\sum^{'\infty }_{n=1}{\left({\mu }_n-{\lambda }_n\right)}=\sum^{'\infty }_{n=1}{\left({\mu }_{n1}-{\lambda }_{n1}\right)}\]
\textit{Thus, the problem is reduced to evaluating a regularized trace of corresponding operator  $L_{11}$.}

 To find the sum on the right side  of (6.32), again apply the technique used above: select function of a complex variable with the poles at ${\beta }_{k,j}$ and residues equal to the terms of the series on the left of side (6.32). Really,  setting
\[
K(z)\equiv-y_k\left(1,\lambda \right)y^{'''}_k\left(1,\lambda \right)-y^{'}_k\left(1,\lambda \right)[y^{''}_k\left(1,\lambda \right)- \lambda \gamma ^{\alpha }_ky^{'}_k\left(1,\lambda \right)]
\]
and in the solution of (5.1),(5.2) defining $c_{2k}$ from condition (5.4) $(y^{''}_k\left(1,\lambda \right)-\ \lambda {\gamma }^{\alpha }_k y^{'}_k\left(1,\lambda \right)=0)$ we have
\begin{equation} 
K^{'}\left({\beta }_{k,j}\right)=\dot{{-y}_k}\left(1,{\beta }^4_{k,j}\right)y^{'''}_k\left(1,{\beta }^4_{k,j}\right)
\end{equation}
So, if
\[F_{1k}\left(z,t\right)=\frac{4z^3y^2_k(t,z^4)}{K(z)}\]
  then in virtue of (6.33) and Theorem 6.1 (or relation 6.29)
\begin{eqnarray*}
{res}_{z={\beta }_{k,j}}F_{1k}\left( z,t\right)  &=&\frac{4{{\beta }_{k,j}}%
^{3}y_{k}^{2}(t,{\beta }_{k,j}^{4})}{K^{^{\prime }}\left( {\beta }%
_{k,j}\right) }=-\frac{4{{\beta }_{k,j}}^{3}y_{k}^{2}(t,{\beta }_{k,j}^{4})}{%
y_{k}^{\prime }\left( 1,{\beta }^4_{k,j}\right) y_{k}^{^{\prime \prime \prime
}}\left( 1,{\beta }^4_{k,j}\right) }= \\
&=&\frac{y_{k}^{2}(t,{\beta }_{k,j}^{4})}{{\left\Vert \Phi _{k,j}\right\Vert
}^{2}}=c_{k,j}^{2}y_{k}^{2}(t,{\beta }_{k,j}^{4})={res}_{z={\beta }%
_{k,j}}F_{k}\left( z,t\right)
\end{eqnarray*}
Now, if we define in the last relation $c_{2k}$ from $y_k\left(1,\lambda \right)=0$, then $F_{1k}\left(z,t\right)$ will be  simplified to the form
\begin{equation} 
F_{1k}\left(z,t\right)=\frac{4z^3y^2_k(t,z^4)}{-2y^{'}_{k}\left(1,z^4\right)[y^{''}_k\left(1,z^4\right)-\ \lambda {\gamma }^{\alpha }_ky^{'}_{k}\left(1,z^4\right)]}
\end{equation}
Thus,
\[\sum^{\infty' }_{n=1}{\left({\mu }_{n1}-{\lambda }_{n1}\right)}=\sum^{\infty }_{k=1}{\sum^{\infty }_{j=1}{\int^1_0{{}^{res}_{z={\beta }_{k,j}}{F_{1k}}(z,t)q_k(t)dt}}}\]
But $F_{1k}\left(z\right)$ together with ${\beta }_{k,j}$ has poles also at the  zeros of the function  $y^{'}_k\left(1,z^4\right). $ Denote them by ${\delta }_{k,j}.$ Thus,
\begin{equation} 
{res}_{z={\delta }_{k,j}}F_{1k}\left(z,t\right)=\frac{4{{\delta }_{k,j}}^3y^2_k(t,z^4)}{-[y^{'}_{k}\left(1,z^4\right)]^{'}|_{z={\delta }_{k,j}}y^{''}_k\left(1,\delta ^4_{k,j}\right)}.
\end{equation}
Again taking the contour $l_N$ ($j=\overline{1,N}$) including ${\beta }_{k,j}$ and ${\delta }_{k,j}$ and extending it to infinity, we will have
 \[
\sum^{\infty }_{j=1}{\int^1_0{{res}_{z=\beta _{k,j}}F_{1k}\left(z,t\right)q_k(t)dt}}=-\sum^{\infty }_{j=1}{\int^1_0{res_{z}=\delta _{k,j}}F_{1k}\left(z,t\right)q_k(t)dt}=
\]
\begin{equation}
=\sum^{\infty }_{j=1}{\frac{\int^1_0{{4{{\delta }_{k,j}}^3 y^2_k(t,z^4)\ }}q_k(t)dt}{[y^{'}_{k}\left(1,\lambda \right)]^{'}|_{z={\delta }_{k,j}}y^{''}_k\left(1,{\delta }^4_{k,j}\right)}},
\end{equation}
where $\frac{-[y^{'}_k\left(1,z^4\right)]^{'}|_{z={\delta }_{k,j}}y^{''}_{k}\left(1,{\delta }^4_{k,j}\right)}{4\delta_{k,j}^3}$ is the norm of orthogonal eigenvectors  of the operator  corresponding to  problem (5.1), (5.2) with the additional conditions
\begin{equation} 
y_k\left(1\right)=0,
\end{equation}
\begin{equation} 
y_k^{'}(1)=0
\end{equation}
or the norm of orthogonal eigenvectors of the operator $L_{02}$. Corresponding perturbed operator  $L_{12}=L_{02}+q(t)$ corresponds to problem (3.1),(3.2) and
\begin{equation} 
y\left(1\right)=0,
\end{equation}
\begin{equation} 
 y^{'}\left(1\right)=0
\end{equation}
in $L_2(H,(0,1))$

 When justifying in (6.36) that
\[c^2_{k,j}=\frac{-4{{\delta }_{k,j}}^3}{[y^{'}_k\left(1,z^4\right)]^{'}|_{z={\delta }_{k,j}}y^{''}_k\left(1,\delta _{k,j}\right)}\]
are, really, norming constants of the operator $L_{02}$ corresponding to (5.1),(5.2), (6.37),(6.38) in $H_3=L_2(H,(0,1))$ 
again use the above technique, but this time  we will not add any additional terms like the term  (6.20) in (6.22) or term  (5.12) in (5.13) (there it was done for defining  the norm in direct sum space, because of $\lambda$ in the boundary conditions. The last boundary conditions don't depend on  $\lambda$ and those condition define a selfadjoint operator in original space).

For not complicating notations denoting the eigenvectors again by $\Phi_{k,j}$ we have by above illustrated  in (5.13)-(5.17) or  (6.20)-(6.26)  technique and defining $c_{2k}$ from (6.37):
\[
\left(z^4-{\delta }^4_{k,j}\right)\int^1_0{y^2_k\left(t,z^4\right)dt= y^{'''}_k}\left(1,z^4\right)y_k\left(1,{\delta }^4_{k,j}\right)-y^{'''}_k\left(1,\delta^{4}_{k,j}\right)y_k\left(1,z^4\right)-
\]
\begin{equation} 
-y^{''}_{k}\left(1,z^4\right)y^{'}_{k}\left(1,{\delta }^4_{k,j}\right)+y^{''}_{k}\left(1,{\delta }^4_{k,j}\right)y^{'}_{k}\left(1,z^4\right)
\end{equation}
Dividing the both sides of this relation by $z-{\delta }_{k,j}$, letting $z\to {\delta }_{k,j}$ , defining $c_{2k}$ from (5.34) and taking into consideration $y^{'}_k\left(1,{\delta }^4_{k,j}\right)$ yields
\begin{equation}
4\delta_{k,j}^{3}\left\Vert \Phi_{k,j}\right\Vert_{3}^{2}=4
\delta_{k,j}^{3}\left\Vert \Psi_{k,j}\right\Vert_{\Lambda _{3}^{2}}=[y_{k}^{\prime}(1, z^{4}) ]^{'}|_{z=\delta_{k,j}}y_{k}^{''}\left(1,\delta_{k,j}^{4}\right)
\end{equation}
 Thus,
\[\sum^{\infty'}_{n=1}{\left(\mu _{n1}-\lambda _{n1}\right)=-\sum^{\infty }_{k=1}\sum^{\infty }_{j=1}\int\limits^{1}_{0} res_{{z}=\delta_{k,j}} }F_{1k}\left(z,t\right)q_k(t)dt=
 \]
\[
=\sum_{j=1}^{\infty }\frac{\int_{0}^{1}4\delta
_{k,j}^{3}y_{k}^{2}(t,z^{4})q_{k}(t)dt}{[y_{k}^{^{^{\prime }}}\left(
1,\lambda \right) ]^{^{\prime }}|_{z={\delta }_{k,j}}y_{k}^{\prime \prime
}\left( 1,\delta _{k,j}^{4}\right) }=\sum_{j=1}^{\infty }\sum_{k=1}^{\infty
}\left( QY_{k,j},Y_{k,j}\right) _{3}
\]
 where $Y_{k,j}$ are orthonormal eigenvectors of the operator $L_{02}$  in $H_3$.

Denoting eigenvalues of $L_{12},L_{02}$ by  ${\mu }_{n2},{\lambda }_{n2}$, respectively,
\[\sum^{\infty' }_{n=1}{\left({\mu }_n-{\lambda }_n\right)}=\sum^{\infty ' }_{n=1}{\left({\mu }_{n1}-{\lambda }_{n1}\right)}=\sum^{\infty ' }_{n=1}{\left({\mu }_{n2}-{\lambda }_{n2}\right)}\]

 Now we come to the evaluation of the sum of series  
\begin{equation}
\sum_{k=1}^{\infty }\sum_{j=1}^{\infty }{{\frac{\int_{0}^{1}{{4{{\delta }%
_{k,j}^{3}}y_{k}^{2}(t,z^{4})\ }}q_{k}(t)dt}{[y^{_{k}^{\prime }}\left(
1,\lambda \right) ]^{\prime }{|_{z={\delta }_{k,j}}y}_{k}^{\prime \prime
}\left( 1,{\delta }^4_{k,j}\right) }}}
\end{equation}

 For that sake select the folowing function of complex variable
\begin{equation} 
F_{2k}\left(z,t\right)=\frac{4z^3y^2_k(t,z^4)}{-y^{'''}_k\left(1,z^4\right){y_k\left(1,z^4\right)+y^{'}_k\left(1,z^4\right)y}^{''}_k\left(1,z^4\right)]}
\end{equation}
whose residues at  ${\delta }_{k,j}$ give terms of series (6.43). Selecting $c_{2k}$  the solution  of boundary value problem from $y_k\left(1\right)=0$, $F_{2k}\left(z,t\right)$ takes the form
\begin{equation}
F_{2k}\left( z,t\right) =\frac{4z^{3}y_{k}^{2}(t,z^{4})}{y_{k}^{\prime }\left(
1,\lambda \right) y_{k}^{\prime \prime }\left( 1,\lambda \right) ]}
\end{equation}%
and
\begin{equation}
{res}_{z={\delta }_{k,j}}F_{2k}\left( z,t\right) =\frac{{4{{\delta }_{k,j}}%
^{3}y_{k}^{2}(t,{{\delta }_{k,j}}^{4})\ }}{[y_{k}^{\prime }\left( 1,{\delta }^4
_{k,j}\right) ]^{\prime }{\ y}_{k}^{\prime \prime }\left( 1,z\right)
^{\prime }|_{z={\delta }_{k,j}}}.
\end{equation}

Obviously $F_{2k}\left(z,t\right)$  will have poles also at roots of the equation $y_k^{''}\left(1,\lambda \right)=0$. Denote these roots by ${\rho }_{k,j}$. Thus, ${\rho }_{k,j}$ are common roots of the equations
\[y_k\left(0,\lambda \right)=0, y_k^{''}\left(0,\lambda \right)=0,\]
\begin{equation} 
\ y_k\left(1,\lambda \right)=0,
\end{equation}
\begin{equation} 
y_k^{''}\left(1,\lambda \right)=0
\end{equation}
moreover,
\begin{equation}
res_{z={\rho }_{k,j}}F_{2k}(z,t)=\frac{4\rho _{k,j}^{3}y_{k}^{2}(t,\rho
_{k,j}^{4})}{\left[ y_{k}^{\prime }(1,\rho^4 _{k,j})\right] ^{\prime
}y_{k}^{\prime \prime }(1,z)^{\prime }|_{z=\rho _{k,j}}}
\end{equation}

 But
\[\frac{-[y^{'}_k\left(1,{\rho }_{k,j}\right)]^{'} y^{''}_k\left(1,z\right)^{'}|_{z={\rho }_{k,j}}}{4\rho _{k,j}^3} =\left\|\Phi_{k,j}\right\|^2_3\]
where $\Phi_{k,j}$ are the eigenvectors of problem  (3.1), (3.2) with additional   boundary conditions
\begin{equation} 
y\left(1\right)=0,y"(1)=0
\end{equation}

 Really,
\[
\left(z^4-{{\rho }^4}_{k,j}\right)\int^1_0{{y_k\left(t,z^4\right)}^2dt=}{y^{'''}_k\left(1,z^4\right)y_k\left(1,{\rho }^4_{k,j}\right)-y}^{'''}_k\left(1,{\rho }^4_{k,j}\right)y_k\left(1,z^4\right)-
\]
\[
-y^{''}_k\left(1,z^4\right)y_k'\left(1,{\rho }^4_{k,j}\right)+y^{''}_k\left(1,{\rho }^4_{k,j}\right)y_k'\left(1,z^4\right)
={-y}^{'''}_k\left(1,{\rho }^4_{k,j}\right)\left[y_k\left(1,z^4\right)-y_k\left(1,{\rho }^4_{k,j}\right)\right]-\]
\begin{equation} 
-y_k'\left(1,{\rho }^4_{k,j}\right)\left[y_k^{''}\left(1,z^4\right)-y_k^{''}\left(1,{\rho }^4_{k,j}\right)\right]
\end{equation}
If  $c_{2k}$ is defined from (6.47), then from (6.51) as  $z\to {\rho }_{k,j}$, 
\begin{equation} 
{{4\rho }^3}_{k,j}{\left\|{\mathrm{\Psi }}_{k,j}\right\|}^2_{\mathrm{\Lambda }}={-y}^{'''}_k\left(1,{\rho }^4_{k,j}\right)y_k\left(1,z^4\right)^{'}|_{z=\rho_{k,j}} -y_k'\left(1,{\rho }^4_{k,j}\right)y_k^{ ''}\left(1,z^4\right)^{'}|_{z=\rho_{k,j}}
\end{equation}
\[{c^2_{k,j}=\left\|{\mathrm{\Psi }}_{k,j}\right\|}^2_{\mathrm{\Lambda }}=\frac{-y_k'\left(1,{\rho }^4_{k,j}\right)y_k"\left(1,z^4\right)^{'}|_{z={\rho }_{k,j}}}{4{{\rho }_{k,j}}^3}\]
Denoting the eigenvalues of $L_{03}$ and $L_{03}+q(t)$ in $L_2(H,(0,1))$ by ${\lambda }_{n3},\; \mu_{n3}$, we come to the next theorem

\textbf{Theorem 6.2.}
 $\sum\limits^{\infty ' }_{n=1}{\left({\mu }_n-{\lambda }_n\right)}=\sum\limits^{\infty ' }_{n=1}{\left({\mu }_{n1}-{\lambda }_{n1}\right)}=\sum\limits^{\infty' }_{n=1}{\left({\mu }_{n2}-{\lambda }_{n2}\right)}$
\[=\sum^{\infty'}_{n=1}\left({\mu }_{n3}-{\lambda }_{n3}\right)
\]

Hence, 
\[\sum\limits^{\infty'}_{n=1}\left(\mu _{n2}-{\lambda }_{n2}\right)=\sum\limits^{\infty }_{k=1}\sum\limits^{\infty }_{j=1}\int^{1}_{0}res_{z=\delta _{k,j}}F_{2k}\left(z,t\right)q_k\left(t\right)dt=
\]
\[
=-\sum\limits^{\infty }_{k=1}\sum\limits^{\infty }_{j=1}\int\limits^{1}_{0} res_{z=\rho _{k,j}}F_{2k}\left(z,t\right)q_k\left(t\right)dt=\sum\limits^{\infty }_{k=1}\sum\limits^{\infty}_{j=1}\left(QY_{k,j},Y_{k,j}\right)_3
\]
where $Y_{k,j}$ are now the set of orthonormal eigenvectors of the operator $L_{03}$ .

 But on the other hand, since the solution satisfying conditions (5.2)  is  given by (5.18), then from (6.47), (6.48)  we have
\[c_{1k}sinz+c_{2k}shz=0\]
\[{-z}^2c_{1k}sinz+c_{2k}z^2shz=0\]
from which $c_{2k}=0$ and orthogonal eigenvectors are $c_{1k}sinzt$

From boundary conditions (6.47),(6.48) follows
$sinz=0$ or $z=\pi j,$ and eigenvalues are  ${\lambda }_{k,j}={(\pi j)}^4+{\gamma }_k$ and orthonormal eigenvectors of $L_{03}$ are $Y_{k,j}=\sqrt{2}sin\pi jt{\varphi }_k,\; k,j=\overline{1,\infty }$

 Thus, taking into consideration also the requirement (6.1)
\[{\sum\limits^{\infty }_{n=1}}^{'}{\left({\mu }_{n3}-{\lambda }_{n3}\right)}=\sum^{\infty }_{k=1}{\sum^{\infty }_{j=1}{{\left(QY_{k,j},Y_{k,j}\right)}_3}}=-\sum^{\infty }_{k=1}{\frac{q_k\left(1\right)+q_k(0)}{4}}\]

\begin{theorem} 
$
\sum\limits_{n=1}^{\infty ^{\prime }}\left( {\mu }_{n}-{\lambda }_{n}\right)
=\sum\limits_{n=1}^{\infty^{\prime }}\left( \mu _{n1}-\lambda _{n1}\right)
=\sum\limits_{n=1}^{\infty^{\prime }}\left( \mu _{n2}-\lambda _{n2}\right)$

\begin{equation} 
={\sum\limits^{\infty }_{n=1}}^{'}{\left({\mu }_{n3}-{\lambda }_{n3}\right)}=-{\sum\limits^{\infty }_{k=1}}^{'}{\frac{q_k\left(1\right)+q_k(0)}{4}}
\end{equation}
If we put on $q(t)$ a stronger condition than (6.1), namely would $q(t)$  belong to the trace class ${\sigma }_1$, then from (6.53) we get 
\end{theorem}

\textbf{Corollary 6.2.}
$ \sum\limits^{\infty'}_{n=1}\left(\mu_n-\lambda _n\right)=\sum\limits^{\infty'}_{n=1}\left(\mu_{n1}-{\lambda }_{n1}\right)=\sum\limits^{\infty' }_{n=1}\left(\mu_{n2}-\lambda _{n2}\right)=
$
\[=\sum\limits^{\infty'}_{n=1}\left(\mu _{n3}-{\lambda }_{n3}\right)=-\frac{trq\left(1\right)+trq(0)}{4}\]

\textbf{Acknowledgement.} The work was supported by Science Development Foundation under the President of the Republic of Azerbaijan-Grant N EIF-ETL-2020-2(36)-16/04/1-M-04

\end{document}